\documentclass[11pt,a4paper]{article}

\usepackage{amsmath,amssymb,amsthm,amsfonts}
\usepackage{mathtools}
\usepackage[margin=1.25in]{geometry}
\usepackage[colorlinks=true,linkcolor=blue,citecolor=blue]{hyperref}
\usepackage{algorithm}       
\usepackage{algpseudocode}   
\usepackage{cleveref}
\usepackage{booktabs}
\usepackage{array}
\usepackage{enumitem}
\usepackage{xcolor}
\usepackage{microtype}
\usepackage{pgfplots}        
\pgfplotsset{compat=1.18}
\usepackage{longtable}       
\usepackage{authblk}         

\newtheorem{theorem}{Theorem}[section]
\newtheorem{lemma}[theorem]{Lemma}
\newtheorem{proposition}[theorem]{Proposition}
\newtheorem{corollary}[theorem]{Corollary}
\newtheorem{definition}[theorem]{Definition}
\newtheorem{assumption}{Assumption}   
\newtheorem{remark}[theorem]{Remark}
\newtheorem{example}[theorem]{Example}

\crefname{assumption}{Assumption}{Assumptions}
\Crefname{assumption}{Assumption}{Assumptions}
\crefname{theorem}{Theorem}{Theorems}
\crefname{lemma}{Lemma}{Lemmas}
\crefname{proposition}{Proposition}{Propositions}
\crefname{corollary}{Corollary}{Corollaries}
\crefname{definition}{Definition}{Definitions}
\crefname{remark}{Remark}{Remarks}
\crefname{example}{Example}{Examples}
\crefname{algorithm}{Algorithm}{Algorithms}
\Crefname{algorithm}{Algorithm}{Algorithms}

\newcommand{\X}{\mathcal{X}}
\newcommand{\G}{\mathcal{G}}
\newcommand{\Omn}{\Omega_N}
\newcommand{\Oms}{\Omega^*}
\newcommand{\Omt}{\Omega^{\mathrm{trans}}}
\newcommand{\calF}{\mathcal{F}}
\newcommand{\R}{\mathbb{R}}
\newcommand{\E}{\mathbb{E}}
\newcommand{\Prob}{\mathbb{P}}
\newcommand{\phist}{\varphi^*}
\newcommand{\Neps}{\mathcal{N}_\varepsilon}
\newcommand{\Akp}{A_k^{+}}
\DeclareMathOperator*{\argmax}{arg\,max}
\DeclareMathOperator{\topN}{top\text{-}N}
\DeclareMathOperator*{\argmin}{arg\,min}
\newcommand{\evals}{\mathcal{E}}
\newcommand{\cache}{\mathcal{C}}

\begin{document}

\title{\bfseries
  Closed-Loop Generative Selection:\\[4pt]
  Convergence, Memory, and Noisy Oracles}

\author[1,3]{Kostantin Fackeldey}
\author[1,2]{Christof Sch\"utte}
\affil[1]{Zuse Institute Berlin, Takustra\ss e~7, 14195 Berlin, Germany}
\affil[2]{Department of Mathematics and Computer Science,
  Freie Universit\"at Berlin, Arnimallee~6, 14195 Berlin, Germany}
\affil[3]{Institute for Mathematics, Technical University Berlin,
Fasanenstr 89,10623 Berlin, Germany}  
\date{June 2026}
\maketitle


   \begin{abstract}
   Closed-loop generative selection has become a workhorse of computational
   drug discovery: a learned generative model proposes candidate molecules, a
   fitness oracle scores them, the best are kept, and the model is retrained on
   this elite set before the next round.  Despite its wide use, the method has
   lacked a rigorous convergence theory, largely because retraining the model
   each round breaks the Markov property on which classical
   evolutionary-algorithm analysis relies.

   We develop a self-contained theory of convergence and expected running time
   for this class of algorithms.  By recovering a Markov structure on an
   enlarged state space, we show that elitism makes the search absorbing, and
   we prove almost-sure convergence together with a runtime bound that
   decomposes the search into the time spent escaping each fitness level.  We
   then analyse the role of the model's memory---how much of the past it is
   trained on.  When learning improves steadily with more data, deeper memory
   never hurts; when it does not, an exit-time analysis pinpoints the optimal
   memory depth and shows that excess memory can actually slow convergence.
   The theory extends to multi-objective search and to noisy oracles: we
   quantify how many repeated evaluations certify progress under light-tailed
   noise, and how robust estimators restore guarantees under heavy tails.
   Recast in terms of oracle evaluations---the true bottleneck in drug
   design---the analysis yields a concrete, evaluation-minimal strategy.  A
   reproducible study confirms the predictions, including the surprising cost
   of excess memory.  We close with three open problems. 
   \end{abstract}


\newpage
\tableofcontents
\newpage

\section{Introduction}
\label{sec:intro}

\paragraph{Context and motivation.}
The discovery of novel molecular entities---candidate drugs, chemical probes,
functional materials---is increasingly driven by closed-loop workflows in
which a machine-learning generative model proposes candidate structures, a
fitness oracle evaluates them, and the model is retrained on the accumulated
data before the next round begins~\cite{Secker2025,Gorgulla2020,Cecchin2026}.
The chemical spaces to be searched are astronomically large: libraries of
make-on-demand compounds routinely exceed $10^{10}$ molecules, and the space
of synthetically accessible drug-like molecules is estimated at $10^{12}$ or
more.  No enumerative or purely random strategy can cover such spaces; the
closed loop is the only practical means to concentrate effort on the most
promising regions.

At its core, every such loop instantiates the same abstract pattern: an
elite pool of $N$ molecules proposes $M$ candidates through a generative
model $G_t$, the candidates are scored against a fitness functional, the
best $N$ of the old pool and the new candidates are retained under an
elitist rule, and $G_t$ is adapted---retrained, fine-tuned, or otherwise
updated---on the resulting elite pool.  We call this pattern \emph{generative
selection}.  The adaptivity of $G_t$ is what distinguishes generative
selection from a classical evolutionary algorithm with a fixed mutation
kernel, and it is the source of both the method's power and its analytical
difficulty.  The generative model learns from the history of elite pools,
concentrating future proposals on promising regions; but analysing this
learning dynamics demands tools that go beyond the standard Markov-chain
theory of fixed-kernel evolutionary algorithms.

\paragraph{Questions addressed.}
We ask the classical questions of randomised-search-heuristic theory in the
form they take for an adaptive generative loop.  Does the algorithm reach an
$\varepsilon$-optimal molecule almost surely, and how many rounds does that
require?  How does the \emph{memory} of the generative model---how much of
the past it is trained on---change the answer?  And, because in drug design
the binding resource is rarely the number of rounds but the number of oracle
evaluations, and because those evaluations are noisy, how many evaluations
are needed, and how does the analysis change when the oracle returns
$\varphi(x)$ corrupted by noise---whether sub-Gaussian or heavy-tailed?

\paragraph{Summary of contributions and roadmap.}
The paper is built from first principles and is self-contained.  Its
contributions fall into two broad groups.

The first group is \emph{oracle-agnostic}: it holds whether or not the
oracle is noisy.  We begin with the correct Markov formulation---because
$G_t$ is retrained each round, the population process $(P_t)$ is in general
not Markov, and we work instead with the augmented chain $(P_t,G_t)$ and
the history process---and establish the absorbing structure that elitism
imposes.  On this foundation we prove almost-sure convergence and a Layered
Hitting-Time Bound that decomposes the expected number of rounds into
per-level escape probabilities (Section~\ref{sec:baseline}).  We then let
the kernel \emph{learn}: for full memory and for finite $m$-step memory the
relevant history process is again Markov, and the hitting-time bound
acquires memory-dependent escape probabilities (Section~\ref{sec:finite}).
Under a natural monotone-learning condition these order into a memory
hierarchy in which more memory never hurts; when that condition fails---as
it does for a naive windowed learner---we replace the hierarchy by an
exit-time theory resolved within each fitness level, which remains valid, is
tighter, and locates the optimal memory depth at the first peak of a
one-dimensional learning profile (Section~\ref{sec:nonmonotone}).  We extend
every result to multi-objective fitness through the hypervolume indicator
(Section~\ref{sec:multiobjective}), and we recast the whole theory in the
currency of oracle evaluations, where the proposal budget $M$ traces a
Pareto frontier between evaluations and rounds whose evaluation-optimal
corner is sequential ($M=1$) operation (Section~\ref{sec:control}).

The second group concerns the \emph{noisy oracle}
(Section~\ref{sec:ctrl-noise}).  When the noise is sub-Gaussian, a molecule
must be certified by averaging repeated queries and bounding the sample
mean; this inflates the evaluation cost by a price-of-noise factor
proportional to $\sigma^2/\Delta^2$ and turns the almost-sure guarantee into
a high-probability one.  When the noise is heavy-tailed, the sample mean is
the wrong tool, and we show how robust-mean estimators (median-of-means,
trimmed means) and, in the extreme, distribution-free sign tests recover
correctness, at a query cost that degrades gracefully with the tail.

\paragraph{Related work and what is new.}
Our methods draw on a long line of work.  The fitness-level method originates
with Wegener~\cite{Wegener2001} and is systematised by
Jansen~\cite{Jansen2013}, with tail-bounded refinements due to
Lehre~\cite{Lehre2010}; drift analysis for hitting times is developed by He
and Yao~\cite{HeYao2001}, Doerr, Johannsen and Winzen~\cite{Doerr2012}, and
Lehre and Yao~\cite{LehreYao2012}, and lower bounds by
Sudholt~\cite{Sudholt2013}.  Estimation-of-distribution algorithms share the
adaptive-generative flavour~\cite{HauschildPelikan2011}, and convergence of
canonical genetic algorithms is treated by Rudolph~\cite{Rudolph1994} and
Schmitt~\cite{Schmitt2004}.  General state-space stability follows Meyn and
Tweedie~\cite{MeynTweedie1993} and finite-chain absorption Kemeny and
Snell~\cite{KemenySnell1960}; the memory-dependent analysis borrows from
online learning~\cite{CesaBianchi2006} and Markov decision
processes~\cite{Puterman1994}, and the heavy-tailed estimation from
Catoni~\cite{Catoni2012} and Lugosi and
Mendelson~\cite{LugosiMendelson2019}.  Four literatures are especially close
to the present work.  First, the runtime analysis of evolutionary algorithms
under \emph{noisy} fitness evaluation---populations as noise
filters~\cite{GiessenKotzing2016}, population-based optimisation of noisy
functions~\cite{DangLehre2015}, and compact genetic algorithms under extreme
Gaussian noise~\cite{Friedrich2017}---parallels our treatment of noisy
oracles, though there the noise corrupts a fixed operator rather than an
adaptive generative model.  Second, our finite-memory window kernels are close
in spirit to \emph{estimation-of-distribution algorithms}, whose runtime
theory has matured through upper~\cite{Witt2019} and lower~\cite{KrejcaWitt2020}
bounds for the univariate marginal distribution algorithm and
significance-based variants~\cite{DoerrKrejca2018}.  Third, replicating noisy
evaluations to certify a level crossing is an instance of \emph{pure
exploration} (fixed-confidence best-arm identification) in stochastic
bandits~\cite{EvenDar2006,Audibert2010,Kaufmann2016}, which is where our
sample-complexity scaling $\sigma^2/\Delta^2$ originates.  Fourth, the
costly-oracle regime connects to \emph{multi-fidelity} surrogate and Bayesian
optimisation~\cite{Forrester2007,Kandasamy2017}, where cheap low-fidelity
evaluations guide expensive high-fidelity ones.  The application of these tools to
elitist evolutionary algorithms with adaptive, possibly noisy generative
models is, to our knowledge, new.  In particular, the explicit analysis of
finite-memory windows as time-homogeneous Markov chains, the
non-monotone exit-time theory that drops the monotone-learning condition,
the evaluation-minimal control strategy, and the treatment of heavy-tailed
noise in this adaptive-loop context are, we believe, original contributions.

\paragraph{Structure of the paper.}
After setting up the mathematical framework (Section~\ref{sec:framework}),
we develop the oracle-agnostic theory---convergence and the layered
hitting-time bound (Section~\ref{sec:baseline}), learning kernels with
memory and the memory hierarchy (Section~\ref{sec:finite}), exit times
without monotone learning (Section~\ref{sec:nonmonotone}), the
multi-objective extension, and the cost of oracle evaluations
(Section~\ref{sec:control})---and then turn to the noisy oracle, treated
under sub-Gaussian (Section~\ref{sec:ctrl-noise}) and then heavy-tailed
noise (Section~\ref{sec:ctrl-heavy}).  We close with concluding remarks
centred on three open problems (Section~\ref{sec:concl}), and an appendix
collecting the notation.

\section{Mathematical Framework}
\label{sec:framework}

\subsection{Chemical space, fitness, and the algorithm}

Let $\X$ be a finite set, the \emph{chemical space}, with $|\X|$
possibly exceeding $10^{12}$.  A scalar \emph{fitness functional}
\[
  \varphi\colon\X\to\R,\qquad
  \phist = \max_{x\in\X}\varphi(x),
\]
encodes the design objective (higher is better).  For a tolerance
$\varepsilon\ge0$ define the \emph{target set}
\[
  \Neps = \{x\in\X : \varphi(x)\ge\phist-\varepsilon\}.
\]
Multi-objective fitness, via the hypervolume indicator, is treated in
Section~\ref{sec:multiobjective}.

A \emph{population} is an $N$-element subset of $\X$; the population
space is
\[
  \Omn = \{P\subseteq\X : |P|=N\}.
\]
Define the \emph{elite fitness} $f(P)=\max_{x\in P}\varphi(x)$.

\paragraph{The generative selection algorithm.}

\begin{algorithm}[ht]
\caption{Generative Selection (one round at time $t$)}\label{alg:gen-sel}
\begin{algorithmic}[1]
\Require elite pool $P_t\in\Omn$; generative model $G_t$; proposal budget $M$
\State \textbf{Propose:} sample $C_t=\{y_1,\dots,y_M\}$ i.i.d.\ from $G_t(\cdot\mid P_t)$
\State \textbf{Score:} evaluate $\varphi(x)$ for every $x\in C_t$
\State \textbf{Select:} $P_{t+1}\leftarrow\topN$ of $P_t\cup C_t$ under $\varphi$ (ties broken arbitrarily)
\State \textbf{Adapt:} form $G_{t+1}$ from the available history
\end{algorithmic}
\end{algorithm}

\medskip
Here $G_t(\cdot\mid P_t)$ is a probability distribution on $\X$
conditioned on the current pool, and $\topN$ returns the $N$ molecules
of highest fitness in $P_t\cup C_t$.  The three settings studied in this
paper differ only in Step~4---how $G_{t+1}$ depends on the past:

\begin{itemize}[nosep]
\item \emph{Static / single-step} (Sections~\ref{sec:static}
  and~\ref{sec:layered}):
  $G_t=G(P_t)$ depends only on the current pool.
\item \emph{Full memory} (Section~\ref{sec:full}):
  $G_t=F(P_0,\dots,P_t)$ depends on the entire history.
\item \emph{Finite $m$-step memory} (Section~\ref{sec:finite}):
  $G_t=G(P_{t-m+1},\dots,P_t)$ depends on the last $m$ pools.
\end{itemize}

\subsection{Elitism and the absorbing structure}

\begin{assumption}[Elitism]\label{ass:elitism}
  Step~3 retains the top-$N$ molecules of $P_t\cup C_t$; consequently
  $f(P_{t+1})\ge f(P_t)$ almost surely for every $t\ge0$.
\end{assumption}

\begin{definition}[Fitness levels]\label{def:levels}
  Let $v_1<v_2<\dots<v_L=\phist$ be the distinct values of $\varphi$ on
  $\X$, and set $L_j=\{x\in\X:\varphi(x)=v_j\}$.  Let
  \[
    K=K(\varepsilon)=\min\{\,j : v_j\ge\phist-\varepsilon\,\}
  \]
  be the index of the first level meeting the $\varepsilon$-target, so
  that $\Neps=\bigcup_{j\ge K}L_j$.  We write $K(\varepsilon)$ to stress
  the dependence on the tolerance; every layered bound below sums to
  $K(\varepsilon)-1$, which is how $\varepsilon$ enters its right-hand side
  (see Remark~\ref{rem:eps-levels}).  For a population $P\in\Omn$ define
  the \emph{level}
  \[
    \ell(P)=\min\bigl(K,\ \max\{\,j : f(P)=v_j\,\}\bigr)\in\{1,\dots,K\},
  \]
  i.e.\ the fitness rank of the best molecule in $P$, capped at the
  target level $K$.  For $k<K$ write $f_k=v_k$ and
  \[
    \Akp=\{y\in\X:\varphi(y)>f_k\}=\bigcup_{j>k}L_j
  \]
  for the set of molecules strictly fitter than level $k$.
\end{definition}

\begin{definition}[Absorbing and transient sets, hitting time]
\label{def:absorbing}
  Let
  \[
    \Oms=\{P\in\Omn:\ell(P)=K\}=\{P:f(P)\ge\phist-\varepsilon\},
    \qquad
    \Omt=\Omn\setminus\Oms,
  \]
  and define the \emph{hitting time}
  $\tau_\varepsilon=\inf\{t\ge0:P_t\in\Oms\}$.
\end{definition}

\begin{lemma}[Monotonicity and absorption]\label{lem:monotone}
  Under Assumption~\ref{ass:elitism}, $\ell(P_t)$ is non-decreasing in
  $t$ almost surely, and $\Oms$ is absorbing: if $P_t\in\Oms$ then
  $P_{s}\in\Oms$ for all $s\ge t$.
\end{lemma}

\begin{proof}
  By Assumption~\ref{ass:elitism}, $f(P_{t+1})\ge f(P_t)$, and
  $\ell(\cdot)$ is a non-decreasing function of $f(\cdot)$ (capped at
  $K$); hence $\ell(P_{t+1})\ge\ell(P_t)$.  If $P_t\in\Oms$ then
  $f(P_t)\ge\phist-\varepsilon$, so $f(P_s)\ge\phist-\varepsilon$ and
  $P_s\in\Oms$ for all $s\ge t$.
\end{proof}

This monotonicity is the bedrock of the entire analysis.  Because the
elite fitness can never decrease, the process is trapped once it enters
the target set, and the hitting time decomposes into a sum of independent
(or at least conditionally independent) sojourns spent at each fitness
level.  The challenge is to bound the duration of each sojourn, and this
is where the adaptive nature of $G_t$ enters.

\subsection{The Markov-chain formulation}
\label{sec:markov}

A subtlety governs the entire analysis: although Step~3 looks like a
Markov transition, the process $(P_t)$ \emph{alone} is generally not a
Markov chain, because the proposal law $G_t$ in Step~1 may depend on the
whole past.  This observation is crucial: standard evolutionary-algorithm
theory, which relies on the Markov property of the population process,
does not directly apply when the generative model is retrained each round.
We therefore record the correct formulations for the three regimes
studied in this paper.

\begin{assumption}[Uniform positivity]\label{ass:positivity}
  There exists $\alpha>0$ such that, for every round $t$, every history
  or window on which $G_t$ may depend, and every $x\in\X$,
  \[
    G_t(x\mid P_t)\ge\alpha.
  \]
  Equivalently $G_t=(1-\alpha|\X|)\,G_t^{\mathrm{learned}}+\alpha|\X|\,U$
  with $U$ uniform on $\X$, i.e.\ the model mixes in a floor of uniform
  exploration.  (Necessarily $\alpha\le1/|\X|$.)
\end{assumption}

This uniform-positivity condition is mild but essential.  It guarantees
that no matter how badly the generative model has learned, every molecule
retains at least an $\alpha$ probability of being proposed, so the
algorithm can never get permanently stuck.  In practice this floor can be
achieved by a softmax temperature, a mixing coefficient, or simply by
retaining a small uniform component in the proposal distribution.

\paragraph{Static kernel: $(P_t)$ is Markov.}
If $G_t=G(P_t)$ depends only on $P_t$, then given $P_t$ the law of
$C_t$, hence of $P_{t+1}=\topN(P_t\cup C_t)$, is determined; thus
$(P_t)_{t\ge0}$ is a time-homogeneous Markov chain on $\Omn$ with kernel
\[
  K(P,A)=\Prob\bigl(\topN(P\cup C)\in A\bigr),
  \qquad C\sim G(\cdot\mid P)^{\otimes M}.
\]

\paragraph{Adaptive kernel: augmented state.}
When $G_t$ is retrained on the history, $(P_t)$ is not Markov.  We
enlarge the state.

\begin{definition}[Augmented state chain]\label{def:augmented}
  Let $\G$ be the space of probability kernels on $\X$.  Suppose the
  adaptation in Step~4 is a deterministic update
  $U\colon\G\times\Omn\to\G$, $G_{t+1}=U(G_t,P_{t+1})$.  Define the
  \emph{augmented state} $Z_t=(P_t,G_t)\in\Omn\times\G$.
\end{definition}

\begin{proposition}[Markov property of the augmented chain]
\label{prop:markov}
  Under Definition~\ref{def:augmented}, $(Z_t)_{t\ge0}$ is a
  time-homogeneous Markov chain:
  $\Prob(Z_{t+1}\in\cdot\mid Z_0,\dots,Z_t)
   =\Prob(Z_{t+1}\in\cdot\mid Z_t)$ a.s.
\end{proposition}

\begin{proof}
  Given $Z_t=(P_t,G_t)$, the law of $C_t\sim G_t(\cdot\mid P_t)^{\otimes
  M}$ is determined, hence so is that of $P_{t+1}=\topN(P_t\cup C_t)$.
  Since $G_{t+1}=U(G_t,P_{t+1})$ is a deterministic function of $Z_t$ and
  $P_{t+1}$, the conditional law of $Z_{t+1}$ depends on the past only
  through $Z_t$.
\end{proof}

The history process formulation, needed when the update is not captured
by a fixed-dimensional $G_t$, is developed in Section~\ref{sec:full}.
All escape probabilities below are defined as infima over the relevant
conditioning information (current pool, history, or window); this makes
the hitting-time arguments valid simultaneously for the Markov marginal
and for the adaptive process.

\subsection{A discrete test problem and the learning kernel}
\label{sec:num-setup}

Throughout the paper we illustrate the theory with a transparent,
reproducible numerical study.  The test problem is chosen not for
realism---it is far simpler than true drug design---but because it
isolates precisely the phenomena the theory predicts, offering a
controlled ``laboratory'' in which to verify the mathematical claims.

\paragraph{Space and oracle.}
Let $\X=\{0,1\}^n$ be the set of binary strings of length $n$, so that
$|\X|=2^n$ plays the role of the astronomically large chemical space.
The fitness oracle is the classical \textsc{OneMax} functional
\[
  \varphi(x)=\sum_{i=1}^n x_i,\qquad \phist=n,
\]
whose unique optimum is the all-ones string.  The fitness levels of
Definition~\ref{def:levels} are the Hamming weight classes
$L_k=\{x:\varphi(x)=k\}$ with $|L_k|=\binom{n}{k}$, and
$\Akp=\{y:\varphi(y)>k\}$ has cardinality $\sum_{j>k}\binom{n}{j}$.  This
is the standard pedagogical landscape for fitness-level
analysis~\cite{Wegener2001,Jansen2013}; here it serves purely as a
controllable stand-in for a molecular objective.

\paragraph{Generative selection with a learning kernel.}
We instantiate $G_t$ as a univariate marginal (estimation-of-distribution)
model~\cite{HauschildPelikan2011}: the kernel maintains a per-coordinate
parameter vector $\theta_t\in[0,1]^n$ and proposes strings
coordinatewise independently,
\[
  G_t(y\mid\cdot)=\prod_{i=1}^n
     \theta_{t,i}^{\,y_i}(1-\theta_{t,i})^{\,1-y_i},
  \qquad
  \theta_{t,i}=\alpha'+(1-2\alpha')\,\widehat\theta_{t,i},
\]
where $\widehat\theta_{t,i}$ is the empirical frequency of bit $i$ over the
elite pools retained in memory and $\alpha'\in(0,\tfrac12)$ is a fixed
floor.  The floor realises the uniform-positivity hypothesis
(Assumption~\ref{ass:positivity}): every coordinate factor is at least
$\alpha'$, so $G_t(y\mid\cdot)\ge(\alpha')^{n}=:\alpha>0$ for all
$y\in\X$.  Selection keeps the top-$N$ strings of $P_t\cup C_t$, so
elitism (Assumption~\ref{ass:elitism}) holds by construction.  The three memory regimes of the paper
correspond exactly to how $\widehat\theta_t$ is formed: the
\emph{single-step} kernel ($m=1$) uses only the current pool, the
\emph{$m$-step} kernel averages the last $m$ pools
(Definition~\ref{def:m-kernel}), and the \emph{full-memory} kernel
(Definition~\ref{def:full}) averages all past pools.  A non-adaptive
\emph{uniform} baseline fixes $\theta_{t,i}\equiv\tfrac12$
($G_t\equiv U$), giving $q_k^0=\alpha|\Akp|$ as in
Corollary~\ref{cor:chain}.

Because both standing hypotheses (Assumptions~\ref{ass:elitism}
and~\ref{ass:positivity}) hold, Theorem~\ref{thm:base-conv} guarantees
almost-sure convergence with a
geometric tail, and Theorem~\ref{thm:base-htb} bounds the expected hitting time
by the layered sum $\sum_k 1/p_k$.  The simulations below confirm both
and probe the role of memory depth.

\paragraph{Reproducibility.}

All curves were produced by a dependency-free Python implementation of
\Cref{alg:gen-sel} (about forty lines using only the standard-library
\texttt{random} module), with fixed seeds; the uniform, single-step,
$m$-step, and full-memory kernels differ only in which elite pools enter
$\widehat\theta_t$.  Figures~\ref{fig:convergence} and~\ref{fig:optdepth} report empirical
medians and means over independent runs; Figure~\ref{fig:hierarchy} is the
closed-form prediction of Proposition~\ref{prop:sojourn-sat}.  Together they confirm
the three qualitative predictions of the theory: almost-sure convergence
of every positive-floor learning kernel, the dominance of learning
kernels over uniform exploration, and the monotone benefit of memory
\emph{exactly when} the monotone-learning condition holds---with a
quantifiable optimal depth otherwise.

\section{Convergence and Expected Hitting Times}
\label{sec:baseline}

This section establishes the basic convergence and expected-hitting-time
theory.  We treat the static and single-step kernels
(Sections~\ref{sec:static} and~\ref{sec:layered}) and then extend the
layered bound to multi-objective fitness via the hypervolume indicator
(Section~\ref{sec:multiobjective}); memory-dependent kernels are the
subject of Section~\ref{sec:finite}.

The core idea is simple: because the elite fitness never decreases, the
algorithm must climb through a finite number of fitness levels to reach
the target.  At each level, the generative model proposes candidates, and
if any of those candidates belongs to a strictly higher level, elitism
pulls the population up.  The expected time spent at each level is
inversely proportional to the probability of such an ``escape.''  The
art lies in bounding these escape probabilities---and, when the generative
model learns, in understanding how they improve with experience.

\subsection{Escape probabilities and almost-sure convergence}
\label{sec:static}

In this subsection and the next, $G_t=G(P_t)$ (static), or more generally
$G_t$ is a possibly time-dependent kernel depending only on the current
pool; all bounds are stated through infima that absorb any such
dependence.

\begin{definition}[Single-proposal escape probability]\label{def:qk}
  For $k=1,\dots,K-1$ define
  \[
    q_k=\inf_{t\ge0}\ \inf_{\substack{P\in\Omn\\\ell(P)=k}}
        G_t\!\bigl(\Akp\mid P\bigr),
    \qquad
    p_k=1-(1-q_k)^M .
  \]
  Here $q_k$ is the worst-case probability that a single proposal from a
  level-$k$ pool is strictly fitter than level $k$, and $p_k$ is the
  corresponding probability that at least one of the $M$ i.i.d.\
  proposals escapes level $k$.
\end{definition}

\begin{remark}
  The escape probability depends on the \emph{population} $P$, not on any
  individual molecule of $P$: there is no optimisation over a chosen
  elite member.  Under Assumption~\ref{ass:positivity},
  \[
    q_k\ \ge\ \alpha\,|\Akp|\ >\ 0
    \qquad\text{for all }k=1,\dots,K-1,
  \]
  since $G_t(\Akp\mid P)=\sum_{y\in\Akp}G_t(y\mid P)\ge\alpha|\Akp|$.
\end{remark}

\paragraph{Almost-sure convergence and a geometric tail.}

Even with the crudest possible bound---using only the uniform exploration
floor $\alpha$---we can already establish that the algorithm converges
almost surely and, moreover, that the tail of the hitting time decays
geometrically.

\begin{theorem}[Convergence and geometric tail]\label{thm:base-conv}
  Under Assumptions~\ref{ass:elitism} and~\ref{ass:positivity}, for
  every initial pool $P_0$,
  \[
    \Prob(\tau_\varepsilon>T)\le(1-\beta)^T\quad(T\ge0),
    \qquad\text{where }\beta=1-(1-\alpha)^M>0,
  \]
  and consequently $\Prob(\tau_\varepsilon<\infty)=1$ and
  $\E[\tau_\varepsilon]\le1/\beta$.
\end{theorem}

\begin{proof}
  Fix any $x^*\in\Neps$.  Conditional on $P_t\in\Omt$, each proposal
  equals $x^*$ with probability $G_t(x^*\mid P_t)\ge\alpha$
  (Assumption~\ref{ass:positivity}), so the probability that none of the
  $M$ i.i.d.\ proposals equals $x^*$ is at most $(1-\alpha)^M$.  If some
  proposal equals $x^*$ then $f(P_{t+1})\ge\varphi(x^*)\ge\phist-\varepsilon$,
  i.e.\ $P_{t+1}\in\Oms$.  Hence
  $\Prob(P_{t+1}\in\Oms\mid\calF_t)\ge\beta$ on $\{P_t\in\Omt\}$.  By
  Lemma~\ref{lem:monotone}, $\Oms$ is absorbing, so
  \[
    \Prob(\tau_\varepsilon>T)
    =\Prob(P_t\in\Omt,\ 0\le t\le T)
    \le(1-\beta)^T,
  \]
  which $\to0$.  Summing the tail gives $\E[\tau_\varepsilon]=
  \sum_{T\ge0}\Prob(\tau_\varepsilon>T)\le\sum_{T\ge0}(1-\beta)^T=1/\beta$.
\end{proof}

This result is reassuring but quantitatively weak: the bound $1/\beta$ can
be astronomically large when $\alpha$ is tiny (as it must be when $|\X|$
is huge).  To obtain sharper, landscape-aware bounds we need the
fitness-level method.

\subsection{The Layered Hitting-Time Bound}
\label{sec:layered}

Theorem~\ref{thm:base-conv} gives a crude bound $1/\beta$ driven by the
exploration floor $\alpha$.  The fitness-level method gives a much
sharper, landscape-aware bound.  The engine is the following domination
lemma, stated generally enough to serve all later sections.

\begin{lemma}[Sojourn domination]\label{lem:sojourn}
  Let $(\calF_t)$ be the natural filtration and suppose that for some
  $p_k\in(0,1]$,
  \begin{equation}
    \Prob\bigl(\ell(P_{t+1})\ge k+1\mid\calF_t\bigr)\ \ge\ p_k
    \qquad\text{a.s.\ on }\{\ell(P_t)=k\}.
    \label{eq:escape-cond}
  \end{equation}
  Define $\tau_k=\inf\{t\ge0:\ell(P_t)\ge k\}$.  Then
  $\E[\tau_{k+1}-\tau_k]\le 1/p_k$.
\end{lemma}

\begin{proof}
  Write $N_k=\tau_{k+1}-\tau_k\ge0$.  On the event
  $\{\tau_k<\infty\}$, for $0\le j<N_k$ the time $\tau_k+j$ has
  $\ell(P_{\tau_k+j})=k$ (the level is $\ge k$ by definition of
  $\tau_k$ and $<k+1$ by definition of $\tau_{k+1}$).  The event
  $\{N_k\ge j\}=\{\tau_k+j\le\tau_{k+1}\}$ is
  $\calF_{\tau_k+j-1}$-measurable, and on it
  \eqref{eq:escape-cond} gives
  $\Prob(N_k\ge j+1\mid\calF_{\tau_k+j-1})\le1-p_k$.  Hence
  $\Prob(N_k\ge j+1)\le(1-p_k)\,\Prob(N_k\ge j)$, so
  $\Prob(N_k\ge j)\le(1-p_k)^{\,j-1}$ for $j\ge1$, and
  \[
    \E[N_k]=\sum_{j\ge1}\Prob(N_k\ge j)
    \le\sum_{j\ge1}(1-p_k)^{j-1}=\frac1{p_k}.
  \]
  (If $\tau_k=\infty$ on a null set only---guaranteed by
  Theorem~\ref{thm:base-conv}---the bound is unaffected.)
\end{proof}

The lemma is simple but powerful: it converts a per-round conditional
escape probability into a bound on the expected sojourn at a level.
Summing over all transient levels yields the main result of this section.

\begin{theorem}[Layered Hitting-Time Bound]\label{thm:base-htb}
  Under Assumptions~\ref{ass:elitism} and~\ref{ass:positivity}, for any
  initial pool $P_0$ with $\ell(P_0)=k_0$,
  \begin{equation}
    \E[\tau_\varepsilon\mid P_0]
    \ \le\ \sum_{k=k_0}^{K(\varepsilon)-1}\frac1{p_k}
    \ =\ \sum_{k=k_0}^{K(\varepsilon)-1}\frac1{1-(1-q_k)^M}.
    \label{eq:base-htb}
  \end{equation}
  In the regime $Mq_k\ll1$ this simplifies to
  $\E[\tau_\varepsilon\mid P_0]\lesssim
  \frac1M\sum_{k=k_0}^{K(\varepsilon)-1}\frac1{q_k}$.
\end{theorem}

\begin{proof}
  By Definition~\ref{def:qk}, on $\{\ell(P_t)=k\}$ each of the $M$
  i.i.d.\ proposals lies in $\Akp$ with probability
  $G_t(\Akp\mid P_t)\ge q_k$, so the probability that at least one does
  is $\ge1-(1-q_k)^M=p_k$; and one such proposal forces
  $\ell(P_{t+1})\ge k+1$ by elitism.  Thus \eqref{eq:escape-cond} holds
  and Lemma~\ref{lem:sojourn} gives
  $\E[\tau_{k+1}-\tau_k]\le1/p_k$.  By Lemma~\ref{lem:monotone} the
  levels are non-decreasing, so
  $\tau_\varepsilon=\tau_K=\sum_{k=k_0}^{K(\varepsilon)-1}(\tau_{k+1}-\tau_k)$
  ($P_0$-a.s.\ finite by Theorem~\ref{thm:base-conv}); summing the
  per-level bounds yields \eqref{eq:base-htb}.  The small-$q$ form uses
  $1-(1-q_k)^M\ge Mq_k(1-q_k)^{M-1}\approx Mq_k$.
\end{proof}

\begin{remark}[Where the accuracy $\varepsilon$ enters]\label{rem:eps-levels}
  The tolerance $\varepsilon$ does not appear explicitly on the right-hand
  side of \eqref{eq:base-htb}---nor of any layered bound in this paper
  (Theorems~\ref{thm:full-htb}, \ref{thm:m-htb}, and~\ref{thm:nm-exit}
  and the evaluation bounds of Section~\ref{sec:control})---because it is
  carried by the \emph{summation range}.  The sum runs over the transient
  levels $k=k_0,\dots,K-1$ with
  $K=K(\varepsilon)=\min\{j:v_j\ge\phist-\varepsilon\}$ the index of the
  first $\varepsilon$-optimal level (Definition~\ref{def:levels}).
  Decreasing $\varepsilon$ raises $K(\varepsilon)$, \emph{appending} the
  harder near-optimal levels (smaller $q_k$), so the bound is monotone in
  $\varepsilon$ and, as $\varepsilon\to0$, runs to the top of the landscape
  and grows---without bound when the top levels have $q_k\to0$.  Thus a
  tighter accuracy is paid for by more, and harder, summands rather than by
  an explicit $\varepsilon$-factor.  Under a noisy oracle $\varepsilon$
  enters a second time, through the certification margin; see
  Remark~\ref{rem:eps}.
\end{remark}

\begin{remark}[Relation to the fitness-level method]\label{rem:fitness-level}
  The bound \eqref{eq:base-htb} is the classical fitness-level (or
  level-based) upper bound~\cite{Wegener2001,Jansen2013} specialised to the
  generative-selection loop.  Partitioning the search space into the level
  sets of Definition~\ref{def:levels} and writing $s_k$ for a lower bound on
  the per-round probability of leaving level $k$ upwards, the fitness-level
  method gives $\E[\tau_\varepsilon]\le\sum_k 1/s_k$; here
  $s_k=p_k=1-(1-q_k)^M$ is supplied by the batch of $M$ i.i.d.\ proposals,
  and the exponent $M$ makes the offspring-population speed-up explicit, in
  the spirit of the parallel fitness-level results of
  Sudholt~\cite{Sudholt2013} and the drift refinements of Doerr, Johannsen
  and Winzen~\cite{Doerr2012}.  Two features distinguish the present setting.
  First, the escape probability $q_k$ is not tied to a fixed mutation
  operator but is a property of the \emph{adaptive} kernel $G_t$, and can
  therefore \emph{improve} as the loop learns---precisely what the memory
  analysis of Section~\ref{sec:finite} quantifies.  Second, in the
  memoryless single-proposal case $M=1$ with $G_t\equiv G$ fixed, one has
  $p_k=q_k$ and \eqref{eq:base-htb} reduces verbatim to the textbook bound
  $\E[\tau_\varepsilon\mid P_0]\le\sum_{k=k_0}^{K(\varepsilon)-1} 1/q_k$.
\end{remark}

\begin{definition}[Landscape difficulty]\label{def:difficulty}
  The quantity $D=\sum_{k=k_0}^{K(\varepsilon)-1}1/q_k$ is the \emph{landscape
  difficulty}; by Theorem~\ref{thm:base-htb},
  $\E[\tau_\varepsilon]\lesssim D/M$ in the small-$q$ regime.  A level
  $k^*$ with $q_{k^*}\ll q_k$ ($k\ne k^*$) is a \emph{bottleneck}
  (``activity cliff''): it dominates $D$ and hence the hitting time.
\end{definition}

\paragraph{Numerical illustration: convergence.}

Figure~\ref{fig:convergence} shows the median best-so-far fitness $f(P_t)$ over
$150$ independent runs at $n=40$, $N=20$, $M=40$.  The non-adaptive
uniform kernel stalls near $f\approx30$: for $k>n/2$ the escape set
$\Akp$ is an exponentially small fraction of $\X$, so $q_k^0$ is
minuscule and the layered bound $\sum_k 1/p_k^0$ is astronomically
large---the algorithm does not reach the optimum within the budget, with
empirical solve rate $0.00$.  Every \emph{learning} kernel, by contrast,
drives $f(P_t)$ to the optimum $\phist=40$ with solve rate $1.00$,
illustrating Theorem~\ref{thm:base-conv}: the adaptive proposal distribution
concentrates on the escape set and converges almost surely.

\begin{figure}[ht]
\centering
\begin{tikzpicture}
\begin{axis}[
    width=0.86\textwidth, height=0.46\textwidth,
    xlabel={round $t$}, ylabel={median best fitness $f(P_t)$},
    xmin=0, xmax=50, ymin=24, ymax=43,
    ytick={24,26,28,30,32,34,36,38,40},
    legend style={at={(0.97,0.55)},anchor=east}, legend cell align=left,
    grid=both, major grid style={black!12}, minor grid style={black!6},
    tick align=outside, thick,
]
\addplot[black, dashed, thin, forget plot] coordinates {(0,40) (50,40)};
\node[anchor=south, font=\footnotesize] at (axis cs:25,40.3) {$\phist=40$ (optimum)};
\addplot[blue, mark=*, mark size=1pt, mark repeat=4] coordinates {
(0,26.0) (1,27.0) (2,29.0) (3,32.0) (4,34.0) (5,35.0) (6,37.0) (7,38.0) (8,39.0) (9,40.0) (10,40.0) (12,40.0) (14,40.0) (16,40.0) (18,40.0) (20,40.0) (24,40.0) (28,40.0) (32,40.0) (36,40.0) (40,40.0) (44,40.0) (48,40.0) (50,40.0)};
\addlegendentry{learning, $m=1$ (single-step)}
\addplot[teal, mark=square*, mark size=1pt, mark repeat=4] coordinates {
(0,26.0) (1,27.0) (2,28.0) (3,29.0) (4,31.0) (5,31.0) (6,32.0) (7,33.0) (8,34.0) (9,35.0) (10,35.0) (11,36.0) (12,37.0) (13,37.0) (14,38.0) (15,39.0) (16,39.0) (17,39.0) (18,40.0) (20,40.0) (24,40.0) (28,40.0) (32,40.0) (36,40.0) (40,40.0) (44,40.0) (48,40.0) (50,40.0)};
\addlegendentry{learning, $m=5$}
\addplot[orange, mark=triangle*, mark size=1.2pt, mark repeat=4] coordinates {
(0,26.0) (2,28.0) (4,31.0) (6,32.0) (8,33.0) (10,34.0) (12,34.0) (14,35.0) (16,35.0) (18,36.0) (20,36.0) (22,36.0) (24,37.0) (26,37.0) (28,37.0) (30,37.0) (32,38.0) (34,38.0) (36,38.0) (38,38.0) (40,38.0) (42,38.0) (44,38.0) (46,39.0) (48,39.0) (50,39.0)};
\addlegendentry{learning, full memory}
\addplot[red, mark=o, mark size=1.2pt, mark repeat=5, densely dashed] coordinates {
(0,26.0) (2,28.0) (4,28.0) (6,29.0) (8,29.0) (10,29.0) (12,29.0) (14,29.0) (16,30.0) (20,30.0) (24,30.0) (28,30.0) (32,30.0) (36,30.0) (40,30.0) (44,30.0) (48,30.0) (50,30.0)};
\addlegendentry{uniform $G_t\equiv U$ (no learning)}
\end{axis}
\end{tikzpicture}
\caption{Overall convergence on \textsc{OneMax} ($n=40$, $N=20$, $M=40$;
  median of $150$ runs).  Learning kernels reach the optimum
  $\phist=40$ (solve rate $1.00$), confirming the almost-sure convergence
  of Theorem~\ref{thm:base-conv}; the non-adaptive uniform kernel stalls
  (solve rate $0.00$) because its high-level escape probabilities
  $q_k^0$ are exponentially small.  Mean hitting times were
  $9.3$ ($m{=}1$), $18.1$ ($m{=}5$), and $69.2$ (full memory) rounds.}
\label{fig:convergence}
\end{figure}
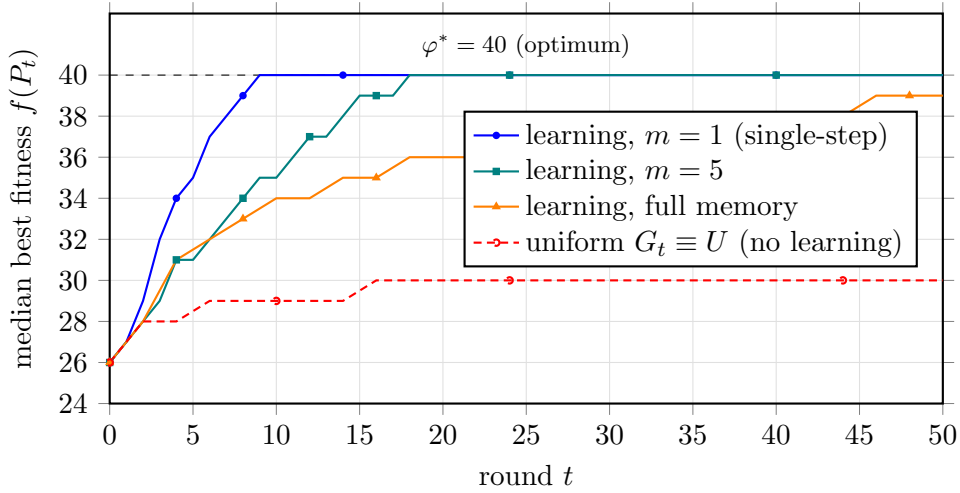

\subsection{Multi-objective fitness via the hypervolume indicator}
\label{sec:multiobjective}

Drug design is intrinsically multi-objective: potency, selectivity,
solubility, synthesizability, and pharmacokinetic properties all matter
simultaneously.  In practice one rarely optimises a single scalar
function; instead, a vector of objectives $\Phi=(\Phi_1,\dots,\Phi_d)$ is
evaluated, and the design seeks molecules that are not dominated on any
component.  We lift the entire theory to this setting via the
\emph{hypervolume indicator}, a standard scalarisation that measures the
volume of objective-space dominated by a population.

\begin{definition}[Hypervolume indicator]\label{def:hv}
  Fix a reference point $r\in\R^d$ with $r_i<\Phi_i(x)$ for all $x\in\X$
  and all $i$.  For $P\in\Omn$,
  \[
    \mathcal{H}(P)=\lambda_d\!\Bigl(\bigcup_{x\in P}[r,\Phi(x)]\Bigr),
    \qquad
    [r,\Phi(x)]=\{y\in\R^d:r_i\le y_i\le\Phi_i(x)\ \forall i\},
  \]
  where $\lambda_d$ is $d$-dimensional Lebesgue measure and $[r,\Phi(x)]$
  is the hyperrectangle with lower corner $r$ and upper corner
  $\Phi(x)$.
\end{definition}

The corner ordering is essential: $[r,\Phi(x)]$ requires $r\preceq
\Phi(x)$, which the reference-point condition guarantees, so each box is
non-degenerate.

\begin{lemma}[Monotonicity of $\mathcal H$]\label{lem:hv-mono}
  If $P\subseteq P'$ then $\mathcal H(P)\le\mathcal H(P')$.  More
  generally, adding a molecule $y$ with $\Phi(y)\not\preceq\Phi(x)$ for
  all $x\in P$ strictly increases $\mathcal H$.
\end{lemma}

\begin{proof}
  $\bigcup_{x\in P}[r,\Phi(x)]\subseteq\bigcup_{x\in P'}[r,\Phi(x)]$ gives
  monotonicity by measure monotonicity.  If $\Phi(y)$ is non-dominated
  by $P$, the box $[r,\Phi(y)]$ contains a point (near corner $\Phi(y)$)
  outside every $[r,\Phi(x)]$, $x\in P$, a set of positive measure.
\end{proof}

\paragraph{Hypervolume fitness levels and convergence.}

Set $f(P)=\mathcal H(P)$ and let $\mathcal H^*=\mathcal
H(\X)=\lambda_d(\bigcup_{x\in\X}[r,\Phi(x)])$ be the hypervolume of the
true Pareto front.  Elitism in Step~3 now retains the $N$ molecules
maximising the hypervolume contribution, so $\mathcal H(P_{t+1})\ge
\mathcal H(P_t)$.  Fix a target tolerance $\varepsilon>0$ and let
\[
  \mathcal N_\varepsilon^{\mathrm{HV}}
  =\{P\in\Omn:\mathcal H(P)\ge\mathcal H^*-\varepsilon\}.
\]
Partition the attainable hypervolume gap into levels
$0=\delta_0<\delta_1<\dots<\delta_K=\mathcal H^*-(\mathcal
H^*-\varepsilon)$ and define the level $\ell^{\mathrm{HV}}(P)$ by which
band $\mathcal H^*-\mathcal H(P)$ falls into (smaller gap $=$ higher
level).  The analogue of the escape set is
\[
  \Akp(\mathrm{HV})
  =\bigl\{y\in\X : \mathcal H(P\cup\{y\})\ge\mathcal H(P)+\Delta f_k\bigr\},
\]
the molecules whose addition raises the hypervolume by at least the
band width $\Delta f_k=\delta_{k+1}-\delta_k$.

\begin{assumption}[Hypervolume uniform positivity]\label{ass:hv-pos}
  There is $\alpha>0$ with $G_t(y\mid P)\ge\alpha$ for every $y\in\X$,
  uniformly over the conditioning information; and for each transient
  level $k$ there is at least one $y$ with $\mathcal
  H(P\cup\{y\})\ge\mathcal H(P)+\Delta f_k$ (attainability of each band).
\end{assumption}

\begin{definition}[Hypervolume escape probabilities]\label{def:hv-qk}
  $q_k^{\mathrm{HV}}=\inf G_t(\Akp(\mathrm{HV})\mid P)$, the infimum over
  the relevant conditioning information among pools with
  $\ell^{\mathrm{HV}}(P)=k$; and $p_k^{\mathrm{HV}}=1-(1-q_k^{\mathrm{HV}})^M$.
\end{definition}

\begin{theorem}[Multi-objective convergence and hitting time]
\label{thm:hv-htb}
  Under Assumption~\ref{ass:elitism} (hypervolume elitism) and
  Assumption~\ref{ass:hv-pos}, the hypervolume process $\mathcal H(P_t)$
  is non-decreasing, $\Prob(\tau_\varepsilon^{\mathrm{HV}}<\infty)=1$
  with geometric tail, and
  \[
    \E[\tau_\varepsilon^{\mathrm{HV}}\mid P_0]
    \ \le\ \sum_{k=k_0}^{K(\varepsilon)-1}\frac1{p_k^{\mathrm{HV}}}
    \ =\ \sum_{k=k_0}^{K(\varepsilon)-1}\frac1{1-(1-q_k^{\mathrm{HV}})^M}.
  \]
\end{theorem}

\begin{proof}
  Monotonicity of $\mathcal H(P_t)$ is Lemma~\ref{lem:hv-mono} under
  hypervolume elitism, so $\ell^{\mathrm{HV}}(P_t)$ is non-decreasing and
  $\mathcal N_\varepsilon^{\mathrm{HV}}$ is absorbing
  (Lemma~\ref{lem:monotone} applies verbatim with $f=\mathcal H$).
  Convergence and the tail follow as in Theorem~\ref{thm:base-conv}:
  some band-raising $y$ exists (Assumption~\ref{ass:hv-pos}) and is
  proposed with probability $\ge\alpha$ each draw.  A band-raising
  proposal forces $\ell^{\mathrm{HV}}(P_{t+1})\ge k+1$ by elitism, so the
  per-round escape probability at level $k$ is $\ge p_k^{\mathrm{HV}}$;
  Lemma~\ref{lem:sojourn} and summation give the bound.
\end{proof}

\paragraph{Memory-dependence of the hypervolume escape probabilities.}

All of Sections~\ref{sec:full}--\ref{sec:comparison} transfer with
$q_k\rightsquigarrow q_k^{\mathrm{HV}}$:

\begin{itemize}[nosep]
\item \emph{Full memory:} define $q_k^{\infty,\mathrm{HV}}$ as in
  Definition~\ref{def:full-qk} with escape set $\Akp(\mathrm{HV})$;
  Theorem~\ref{thm:full-htb} holds verbatim.
\item \emph{Finite memory:} the window process $S_t^m$ is Markov
  (Theorem~\ref{thm:m-markov} does not use the form of $\varphi$), and
  Theorem~\ref{thm:m-htb} holds with $q_k^{m,\mathrm{HV}}$.
\item \emph{Monotonicity:} under Assumption~\ref{ass:mono} stated for
  $\Akp(\mathrm{HV})$, Theorem~\ref{thm:mono} gives
  $q_k^{m,\mathrm{HV}}\ge q_k^{m-1,\mathrm{HV}}$, hence the full
  hypervolume memory hierarchy.
\end{itemize}

\begin{remark}[Pareto-rank variant]
  An alternative level structure uses non-domination rank: define
  $\ell(P)$ via the best Pareto rank present in $P$ and $\Akp$ as the
  molecules of strictly better rank.  Because non-domination rank is
  also monotone under elitist selection, Lemmas~\ref{lem:monotone}
  and~\ref{lem:sojourn} apply unchanged and all hitting-time and
  memory-hierarchy results follow identically.
\end{remark}

\section{Learning Kernels with Memory}
\label{sec:finite}

This section studies kernels that \emph{learn} from the history of elite
pools: first the idealised \emph{full-memory} case, in which $G_t$ depends
on the entire trajectory, and then the practically important \emph{finite
$m$-step memory}, in which $G_t$ uses only the last $m$ pools.

The central insight is that memory---how much of the past the generative
model is trained on---creates a Markov chain on an appropriately enlarged
state space.  For full memory, the history process itself is Markov; for
finite memory, the sliding window of the last $m$ pools is a
time-homogeneous Markov chain.  This structural observation lets us apply
the same sojourn-domination technique developed in Section~\ref{sec:layered},
but now with escape probabilities that depend on the memory depth.

\subsection{Full-memory learning kernels}
\label{sec:full}

We now let the generative model depend on the \emph{entire} history of
elite pools---the idealised limit of retraining on all data discovered
so far.  While impractical when the number of rounds grows large, this
setting serves as a theoretical benchmark against which finite-memory
algorithms can be measured.

\paragraph{The history process is Markov.}

\begin{definition}[Full-memory kernel]\label{def:full}
  Let $F\colon\bigcup_{t\ge0}\Omn^{t+1}\to\G$ be measurable.  The
  \emph{full-memory} algorithm runs \Cref{alg:gen-sel} with
  $G_t=F(P_0,\dots,P_t)$.  The \emph{history process} is
  $\mathbf{H}_t=(P_0,\dots,P_t)\in S_\infty:=\bigcup_{n\ge1}\Omn^n$.
\end{definition}

\begin{proposition}[Markov property of the history process]
\label{prop:history-markov}
  Under the full-memory algorithm, $(\mathbf{H}_t)_{t\ge0}$ is a Markov
  chain on $S_\infty$.
\end{proposition}

\begin{proof}
  $\mathbf{H}_{t+1}=(\mathbf{H}_t,P_{t+1})$.  Given
  $\mathbf{H}_t=h=(P_0,\dots,P_t)$, the law of $P_{t+1}$ is determined by
  $P_t=\mathrm{last}(h)$ and $G_t=F(h)$, both functions of $h$.  Hence
  the conditional law of $\mathbf{H}_{t+1}$ depends on the past only
  through $\mathbf{H}_t$.
\end{proof}

\paragraph{Convergence and the layered bound.}

\begin{assumption}[Full-memory uniform positivity]\label{ass:full-pos}
  There exists $\alpha>0$ such that for all $t$, all $h\in\Omn^{t+1}$,
  and all $x\in\X$: $F(h)(x\mid\mathrm{last}(h))\ge\alpha$.
\end{assumption}

\begin{theorem}[Convergence and geometric tail, full memory]
\label{thm:full-conv}
  Under Assumptions~\ref{ass:elitism} and~\ref{ass:full-pos},
  $\Prob(\tau_\varepsilon>T)\le(1-\beta)^T$ with
  $\beta=1-(1-\alpha)^M$; in particular $\Prob(\tau_\varepsilon<\infty)=1$.
\end{theorem}

\begin{proof}
  Identical to Theorem~\ref{thm:base-conv}, conditioning on
  $\mathbf{H}_t=h$ in place of $P_t$ and using
  $F(h)(x^*\mid\mathrm{last}(h))\ge\alpha$ from
  Assumption~\ref{ass:full-pos}; the bound is uniform over histories with
  $\mathrm{last}(h)\in\Omt$.
\end{proof}

\begin{definition}[Full-memory escape probabilities]\label{def:full-qk}
  For $k=1,\dots,K-1$,
  \[
    q_k^\infty=\inf_{t\ge0}\
      \inf_{\substack{h\in\Omn^{t+1}\\\ell(\mathrm{last}(h))=k}}
      F(h)\!\bigl(\Akp\mid\mathrm{last}(h)\bigr),
    \qquad p_k^\infty=1-(1-q_k^\infty)^M.
  \]
\end{definition}

Under Assumption~\ref{ass:full-pos}, $q_k^\infty\ge\alpha|\Akp|>0$.

\begin{theorem}[Layered Hitting-Time Bound, full memory]
\label{thm:full-htb}
  Under Assumptions~\ref{ass:elitism} and~\ref{ass:full-pos}, for any
  $P_0$ with $\ell(P_0)=k_0$,
  \[
    \E[\tau_\varepsilon\mid P_0]
    \ \le\ \sum_{k=k_0}^{K(\varepsilon)-1}\frac1{p_k^\infty}
    \ =\ \sum_{k=k_0}^{K(\varepsilon)-1}\frac1{1-(1-q_k^\infty)^M}.
  \]
\end{theorem}

\begin{proof}
  By Definition~\ref{def:full-qk}, on $\{\ell(P_t)=k\}$ the conditional
  escape probability given $\calF_t$ is $\ge p_k^\infty$ regardless of
  the history $\mathbf{H}_t$.  Lemma~\ref{lem:sojourn} (with
  $p_k=p_k^\infty$) and summation over levels, exactly as in
  Theorem~\ref{thm:base-htb}, give the bound.
\end{proof}

\begin{remark}[Gain over the non-adaptive baseline]
  Let $q_k^0=\alpha|\Akp|$ be the escape probability of pure uniform
  exploration ($G_t\equiv U$).  If learning never hurts, $q_k^\infty\ge
  q_k^0$, so $\E[\tau_\varepsilon\mid\text{full memory}]\le
  \E[\tau_\varepsilon\mid\text{baseline}]$; the speed-up is captured by
  the ratios $q_k^\infty/q_k^0$.
\end{remark}

\paragraph{From regret to sojourn time.}

When the generative model is an online learner, its regret while sitting
at a level converts directly into a sojourn bound.  The next theorem is
\emph{exact} (no small-$q$ approximation): it uses only the identity
$1-p_k^{(s)}=(1-q_k^{(s)})^M$ and $\ln(1-q)\le-q$.

\begin{definition}[Per-round escape sequence at level $k$]
  While the chain sits at level $k$, index the rounds by $s=0,1,2,\dots$
  and let $q_k^{(s)}=G_{t}(\Akp\mid P_t)$ be the single-proposal escape
  probability at the $s$-th such round.  Let $q_k^\infty$ be the
  saturated value of Definition~\ref{def:full-qk}.
\end{definition}

\begin{theorem}[Regret-to-sojourn conversion]\label{thm:online}
  Suppose $q_k^{(s)}\ge q_k^\infty(1-r_s)$ for a sequence $r_s\ge0$ with
  finite total regret $R_\infty=\sum_{s\ge0}r_s<\infty$.  Then
  \begin{equation}
    \E[\tau_{k+1}-\tau_k]
    \ \le\
    \frac{e^{Mq_k^\infty R_\infty}}{\,1-e^{-Mq_k^\infty}\,}.
    \label{eq:online}
  \end{equation}
\end{theorem}

\begin{proof}
  The sojourn equals $\sum_{s\ge0}\prod_{j=0}^{s-1}(1-p_k^{(j)})$.  Since
  $1-p_k^{(j)}=(1-q_k^{(j)})^M$ exactly,
  \[
    \prod_{j=0}^{s-1}(1-p_k^{(j)})
    =\exp\!\Bigl(M\sum_{j=0}^{s-1}\ln(1-q_k^{(j)})\Bigr)
    \le\exp\!\Bigl(-M\sum_{j=0}^{s-1}q_k^{(j)}\Bigr),
  \]
  using $\ln(1-q)\le-q$.  By hypothesis
  $\sum_{j=0}^{s-1}q_k^{(j)}\ge q_k^\infty\bigl(s-\sum_{j=0}^{s-1}r_j\bigr)
  \ge q_k^\infty(s-R_\infty)$, so
  \[
    \E[\tau_{k+1}-\tau_k]
    \le\sum_{s\ge0}e^{-Mq_k^\infty(s-R_\infty)}
    =e^{Mq_k^\infty R_\infty}\sum_{s\ge0}e^{-Mq_k^\infty s}
    =\frac{e^{Mq_k^\infty R_\infty}}{1-e^{-Mq_k^\infty}}.\qedhere
  \]
\end{proof}

\begin{remark}[Interpretation]\label{rem:online}
  As $Mq_k^\infty\to0$ with $R_\infty$ fixed, the right-hand side of
  \eqref{eq:online} equals $\frac1{Mq_k^\infty}(1+O(Mq_k^\infty))$,
  recovering the optimal sojourn $1/p_k^\infty\approx1/(Mq_k^\infty)$ of
  a learner that has already converged.  The factor $e^{Mq_k^\infty
  R_\infty}$ is the multiplicative penalty for the transient learning
  phase.  Finiteness of $R_\infty$ holds for any online method with
  per-step regret $r_s=O(s^{-1-\delta})$, $\delta>0$ (e.g.\ logarithmic
  cumulative regret); the borderline $O(\sqrt T)$ regime
  ($r_s\sim s^{-1/2}$) is \emph{not} summable and is flagged as
  a question we leave open.  We caution that
  $1-e^{-x}\le x$, so one may \emph{not} replace the denominator of
  \eqref{eq:online} by $Mq_k^\infty$; the correct bound is
  \eqref{eq:online} itself.
\end{remark}

The regret-to-sojourn conversion provides a direct bridge between the
online-learning literature and the hitting-time analysis of generative
selection.  It quantifies exactly how much the transient learning phase
slows convergence relative to the idealised saturated learner.

Full memory is an idealisation; in practice the model is retrained on a
sliding window of the last $m$ elite pools.  The remainder of this section
shows that a window of depth $m$ yields a \emph{time-homogeneous} Markov
chain on $\Omn^m$, and quantifies how the hitting time improves with $m$.

\subsection{The window chain and its hitting-time bound}

\begin{definition}[$m$-step memory kernel]\label{def:m-kernel}
  Fix $m\ge1$ and a measurable $G\colon\Omn^m\to\G$.  The \emph{$m$-step}
  algorithm runs \Cref{alg:gen-sel} with
  $G_t=G\bigl(P_{\max(0,t-m+1)},\dots,P_t\bigr)$; for $t\ge m-1$ this is
  $G(P_{t-m+1},\dots,P_t)$.  The \emph{window process} is
  \[
    S_t^m=(P_{t-m+1},\dots,P_t)\in\Omn^m,\qquad t\ge m-1.
  \]
\end{definition}

\begin{theorem}[$m$-step Markov property]\label{thm:m-markov}
  For $t\ge m-1$, $(S_t^m)$ is a time-homogeneous Markov chain on
  $\Omn^m$.
\end{theorem}

\begin{proof}
  Write $S_t^m=(Q_1,\dots,Q_m)$ with $Q_m=P_t$.  Step~3 gives
  $P_{t+1}=\topN(Q_m\cup C_t)$, $C_t\sim G(S_t^m)(\cdot\mid Q_m)^{\otimes
  M}$, so $P_{t+1}$ is conditionally independent of
  $(P_0,\dots,P_{t-m})$ given $S_t^m$.  The next window
  $S_{t+1}^m=(Q_2,\dots,Q_m,P_{t+1})$ is a deterministic function of
  $S_t^m$ and $P_{t+1}$, and the transition law is the same for all
  $t\ge m-1$ because $G$ has no explicit time dependence.
\end{proof}

\begin{remark}
  At $m=1$, $S_t^1=P_t$ and we recover the single-step chain of
  Sections~\ref{sec:static} and~\ref{sec:layered}.  As $m$ grows, more of
  the learning history is encoded in the state; $m\to\infty$ recovers the
  full-memory setting of Section~\ref{sec:full}.
\end{remark}

\paragraph{Convergence and the layered bound.}

\begin{assumption}[$m$-step uniform positivity]\label{ass:m-pos}
  There is $\alpha>0$ with $G(S)(x\mid Q_m)\ge\alpha$ for all
  $S=(Q_1,\dots,Q_m)\in\Omn^m$ and all $x\in\X$.
\end{assumption}

\begin{definition}[$m$-step escape probabilities]\label{def:m-qk}
  For $k=1,\dots,K-1$,
  \[
    q_k^m=\inf_{\substack{S=(Q_1,\dots,Q_m)\in\Omn^m\\\ell(Q_m)=k}}
          G(S)\!\bigl(\Akp\mid Q_m\bigr),
    \qquad
    p_k^m=1-(1-q_k^m)^M .
  \]
\end{definition}

\begin{theorem}[Convergence and Layered Bound, $m$-step memory]
\label{thm:m-htb}
  Under Assumptions~\ref{ass:elitism} and~\ref{ass:m-pos}:
  $\Prob(\tau_\varepsilon<\infty)=1$ with geometric tail
  $(1-\beta)^T$; and for any $P_0$ with $\ell(P_0)=k_0$,
  \[
    \E[\tau_\varepsilon\mid P_0]
    \ \le\ \sum_{k=k_0}^{K(\varepsilon)-1}\frac1{p_k^m}
    \ =\ \sum_{k=k_0}^{K(\varepsilon)-1}\frac1{1-(1-q_k^m)^M}.
  \]
\end{theorem}

\begin{proof}
  Convergence and the tail follow as in Theorem~\ref{thm:base-conv} with
  $\alpha$ from Assumption~\ref{ass:m-pos}.  By
  Definition~\ref{def:m-qk}, on $\{\ell(P_t)=k\}$ the conditional escape
  probability given $\calF_t$ is $\ge p_k^m$ for every window with
  $\ell(Q_m)=k$; Lemma~\ref{lem:sojourn} and summation give the bound.
\end{proof}

\subsection{Monotonicity in memory depth}

The central structural result of this section is that, under a natural
condition, more memory never hurts: the escape probabilities are
non-decreasing in the memory depth $m$, and therefore the expected hitting
time is non-increasing.

\begin{assumption}[Monotone learning]\label{ass:mono}
  For all $m\ge2$, all $(Q_1,\dots,Q_m)\in\Omn^m$, and all $k$,
  \[
    G(Q_1,\dots,Q_m)\!\bigl(\Akp\mid Q_m\bigr)
    \ \ge\
    G(Q_2,\dots,Q_m)\!\bigl(\Akp\mid Q_m\bigr).
  \]
  That is, conditioning on one more past pool does not decrease the
  probability of proposing above level $k$.
\end{assumption}

\begin{remark}
  Assumption~\ref{ass:mono} concerns the \emph{learner}, not the
  landscape.  It holds for Bayesian posteriors (more data cannot worsen
  expected predictive performance under a proper prior), for
  maximum-likelihood estimators with monotone sufficient statistics, and
  for fine-tuned models trained on more samples from a higher-quality
  elite distribution.  When it holds, we can prove a clean monotonicity
  theorem; when it fails, the analysis of Section~\ref{sec:nonmonotone}
  applies.
\end{remark}

\begin{theorem}[Monotonicity in $m$]\label{thm:mono}
  Under Assumption~\ref{ass:mono}, $q_k^m\ge q_k^{m-1}$ for all $m\ge2$
  and all $k$.  Hence $p_k^m\ge p_k^{m-1}$ and
  \[
    \E[\tau_\varepsilon\mid m\text{-step}]
    \ \le\
    \E[\tau_\varepsilon\mid(m-1)\text{-step}].
  \]
\end{theorem}

\begin{proof}
  Fix $S=(Q_1,\dots,Q_m)$ with $\ell(Q_m)=k$.  By
  Assumption~\ref{ass:mono},
  \[
    G(Q_1,\dots,Q_m)(\Akp\mid Q_m)
    \ge G(Q_2,\dots,Q_m)(\Akp\mid Q_m)
    \ge q_k^{m-1},
  \]
  the last step because $(Q_2,\dots,Q_m)$ is an $(m-1)$-window with
  $\ell(Q_m)=k$ and $q_k^{m-1}$ is the infimum over all such windows.
  Taking the infimum over $S$ gives $q_k^m\ge q_k^{m-1}$.  Monotonicity
  of $p_k^m$ in $q_k^m$ and of $1/p_k^m$ complete the proof.
\end{proof}

\begin{corollary}[Monotone chain of escape probabilities]\label{cor:chain}
  With $q_k^0=\alpha|\Akp|$ (uniform-exploration baseline), under
  Assumptions~\ref{ass:m-pos} and~\ref{ass:mono},
  \[
    q_k^0\le q_k^1\le q_k^2\le\cdots\le q_k^m\ \nearrow\ q_k^\infty
    \quad(m\to\infty),
  \]
  and correspondingly for $p_k^m$ and for the hitting-time bounds.
\end{corollary}

\paragraph{Saturation model and closed-form sojourn estimates.}

To make Theorem~\ref{thm:mono} quantitative we posit a parametric law
for the growth of $q_k^m$ with $m$.

\begin{definition}[Exponential saturation model]\label{def:sat}
  For each level $k$ there exist $q_k^0>0$, $q_k^\infty\ge q_k^0$, and a
  \emph{learning rate} $c_k>0$ with
  \begin{equation}
    q_k^m=q_k^\infty-(q_k^\infty-q_k^0)\,e^{-c_k m},\qquad m\ge0.
    \label{eq:sat}
  \end{equation}
\end{definition}

In the saturation regime the escape probability when the window holds
$j$ level-$k$ pools is $q_k^j$; thus during a sojourn at level $k$ the
$s$-th round (counting from entry) has effective escape probability
$q_k^{\min(s,m)}$, since after $m$ rounds the window is saturated with
level-$k$ data.

\begin{proposition}[Closed-form sojourn under saturation]
\label{prop:sojourn-sat}
  Under \eqref{eq:sat} and the windowing just described, the expected
  sojourn at level $k$ with $m$-step memory is
  \begin{equation}
    \E[\tau_{k+1}-\tau_k]
    \ =\ A_k(m)+\frac{B_k(m)}{p_k^m},
    \label{eq:sojourn-sat}
  \end{equation}
  where, with $p_k^j=1-(1-q_k^j)^M$,
  \[
    A_k(m)=\sum_{s=0}^{m-1}\prod_{j=0}^{s-1}(1-p_k^j)\le m,
    \qquad
    B_k(m)=\prod_{j=0}^{m-1}(1-p_k^j)\le(1-p_k^0)^m.
  \]
  Consequently:
  \emph{(a)} $B_k(m)\to0$ exponentially and
  $\E[\tau_{k+1}-\tau_k]\to1/p_k^\infty$ as $m\to\infty$;
  \emph{(b)} at $m=0$, $A_k(0)=0$, $B_k(0)=1$, giving $1/p_k^0$;
  \emph{(c)} the saving $\Delta_k(m)=\tfrac1{p_k^0}-A_k(m)-B_k(m)/p_k^m\ge0$
  is non-decreasing in $m$.
\end{proposition}

\begin{proof}
  The escape probability at round $s$ is $p_k^{\min(s,m)}$, so for
  $s\ge m$ it is constant at $p_k^m$.  Hence
  \[
    \E[\tau_{k+1}-\tau_k]
    =\sum_{s=0}^{\infty}\prod_{j=0}^{s-1}(1-p_k^{\min(j,m)})
    =\underbrace{\sum_{s=0}^{m-1}\prod_{j=0}^{s-1}(1-p_k^j)}_{A_k(m)}
     +\underbrace{\prod_{j=0}^{m-1}(1-p_k^j)}_{B_k(m)}
       \sum_{i=0}^{\infty}(1-p_k^m)^i,
  \]
  and the geometric sum is $1/p_k^m$.  The bound $A_k(m)\le m$ holds
  since each summand $\le1$; $B_k(m)\le(1-p_k^0)^m$ holds because
  $p_k^j\ge p_k^0$ (Corollary~\ref{cor:chain}).  Parts (a)--(c) are
  immediate.
\end{proof}

\paragraph{Explicit decay.}
In the small-$q$ regime $p_k^j\approx Mq_k^j$ and, using
$1-p_k^j=(1-q_k^j)^M$,
\[
  B_k(m)\le\exp\!\Bigl(-M\sum_{j=0}^{m-1}q_k^j\Bigr)
  =\exp\!\Bigl(-Mmq_k^\infty
    +\tfrac{M(q_k^\infty-q_k^0)(1-e^{-c_km})}{1-e^{-c_k}}\Bigr),
\]
so $B_k(m)$ decays like $e^{-Mmq_k^\infty}$ and the $m$-step bound
approaches the full-memory bound $\sum_k1/p_k^\infty$ exponentially fast
in $m$.

\subsection{A worked example: relief of a single bottleneck}
\label{sec:example}

We illustrate the memory hierarchy with a stylised landscape that
isolates the effect of learning at one bottleneck level.

\begin{example}\label{ex:bottleneck}
  Suppose all levels except one are easy: $q_k^m\approx1$ for $k\ne k^*$,
  so their sojourns are negligible, and the hitting time is governed by
  the bottleneck $k^*$.  Take a proposal budget $M=100$ and a saturation
  model \eqref{eq:sat} at $k^*$ with
  \[
    q_{k^*}^0=10^{-3},\qquad q_{k^*}^\infty=10^{-2},\qquad c_{k^*}=0.5 .
  \]
  Using $q_{k^*}^m=q_{k^*}^\infty-(q_{k^*}^\infty-q_{k^*}^0)e^{-0.5m}$ and
  $p_{k^*}^m=1-(1-q_{k^*}^m)^M$, the saturated per-round escape
  probability and the expected bottleneck sojourn $1/p_{k^*}^m$ are:

  \begin{center}
  \renewcommand{\arraystretch}{1.25}
  \begin{tabular}{@{}cccc@{}}
  \toprule
  memory $m$ & $q_{k^*}^m$ & $p_{k^*}^m=1-(1-q_{k^*}^m)^{100}$
    & sojourn $1/p_{k^*}^m$ \\
  \midrule
  $0$        & $0.00100$ & $0.0952$ & $10.51$ \\
  $1$        & $0.00454$ & $0.366$  & $2.74$  \\
  $2$        & $0.00669$ & $0.489$  & $2.05$  \\
  $4$        & $0.00878$ & $0.586$  & $1.71$  \\
  $8$        & $0.00984$ & $0.628$  & $1.59$  \\
  $\infty$   & $0.01000$ & $0.632$  & $1.58$  \\
  \bottomrule
  \end{tabular}
  \end{center}

  The no-memory algorithm needs $\approx10.5$ rounds to clear the
  bottleneck; a single step of memory already cuts this to $\approx2.7$
  rounds, and the saturated value $\approx1.58$ rounds is essentially
  reached by $m=8$. The overall speed-up from memory at this bottleneck
  is a factor $10.51/1.58\approx6.6$. This is exactly the
  diminishing-returns structure of
  Proposition~\ref{prop:sojourn-sat}: the marginal benefit of each
  additional memory step decays geometrically, so a modest window
  ($m\approx 5$--$8$ here) captures almost all of the achievable
  acceleration. The same numbers, read through
  Corollary~\ref{cor:bottleneck}, bound the memory-adjusted difficulty
  $1/q_{k^*}^m$ falling from $1000$ to $100$.
\end{example}

\subsection{The memory hierarchy}
\label{sec:comparison}

\begin{theorem}[Memory hierarchy]\label{thm:hierarchy}
  Under Assumptions~\ref{ass:elitism}, \ref{ass:full-pos},
  \ref{ass:m-pos}, and~\ref{ass:mono}, for all $m\ge1$,
  \[
    \underbrace{\sum_k\frac1{p_k^\infty}}_{\text{full memory}}
    \le\cdots\le
    \underbrace{\sum_k\frac1{p_k^m}}_{m\text{-step}}
    \le
    \underbrace{\sum_k\frac1{p_k^{m-1}}}_{(m-1)\text{-step}}
    \le\cdots\le
    \underbrace{\sum_k\frac1{p_k^1}}_{\text{single-step}}
    \le
    \underbrace{\sum_k\frac1{p_k^0}}_{\text{no memory}} .
  \]
\end{theorem}

\begin{proof}
  Each inequality is Theorem~\ref{thm:mono} together with monotonicity of
  $1/p_k^m$ in $q_k^m$; the leftmost term is
  Theorem~\ref{thm:full-htb} and the rightmost is the
  uniform-exploration baseline $q_k^0=\alpha|\Akp|$.
\end{proof}

\begin{theorem}[Convergence of the $m\to\infty$ limit]\label{thm:limit}
  If $q_k^m\nearrow q_k^\infty$ for every $k$ (e.g.\ under
  \eqref{eq:sat}), then
  \[
    \lim_{m\to\infty}\sum_{k=k_0}^{K(\varepsilon)-1}\frac1{1-(1-q_k^m)^M}
    =\sum_{k=k_0}^{K(\varepsilon)-1}\frac1{1-(1-q_k^\infty)^M}.
  \]
\end{theorem}

\begin{proof}
  Each summand is positive, non-increasing in $m$, and continuous in
  $q_k^m$; as the sum is finite ($K-k_0$ terms), the limit passes inside.
\end{proof}

\begin{definition}[Memory-adjusted difficulty]
  $D^m=\sum_{k=k_0}^{K(\varepsilon)-1}1/q_k^m$.  By Theorem~\ref{thm:mono},
  $D^m\le D^{m-1}$: memory reduces effective landscape difficulty.
\end{definition}

\begin{corollary}[Bottleneck relief]\label{cor:bottleneck}
  Under \eqref{eq:sat}, at a bottleneck level $k^*$
  (Definition~\ref{def:difficulty}),
  \[
    \frac1{q_{k^*}^m}
    \ \le\
    \frac{e^{c_{k^*}m}}
         {q_{k^*}^\infty(e^{c_{k^*}m}-1)+q_{k^*}^0}
    \ \xrightarrow[m\to\infty]{}\ \frac1{q_{k^*}^\infty}.
  \]
  Thus a single learnable bottleneck, which dominates the no-memory
  hitting time, is relieved at geometric rate $c_{k^*}$ in $m$.
\end{corollary}

\begin{proof}
  Rearranging \eqref{eq:sat},
  $q_{k^*}^m=\dfrac{q_{k^*}^\infty(e^{c_{k^*}m}-1)+q_{k^*}^0}{e^{c_{k^*}m}}$;
  invert.
\end{proof}

\subsection{Memory depth: a bias--variance trade-off}
\label{sec:num-memory}

Figure~\ref{fig:convergence} already shows a striking feature: \emph{more}
memory converges \emph{slower} here.  This is not a contradiction of the
memory hierarchy (Theorem~\ref{thm:hierarchy}); it is a failure of its
hypothesis.  The hierarchy assumes \emph{monotone learning}
(Assumption~\ref{ass:mono})---that conditioning on an additional past pool never
lowers the escape probability.  Under elitism the elite fitness is
non-decreasing, so on \textsc{OneMax} older pools have systematically
\emph{lower} marginals; averaging them into $\widehat\theta_t$ biases the
proposal distribution toward worse strings, and the plain windowed mean
\emph{violates} Assumption~\ref{ass:mono}.  The simulation thus delineates
exactly when memory helps.

To expose the trade-off we make a single pool an unreliable summary of
the landscape by (i) shrinking the pool to $N=6$, $M=12$ and (ii)
corrupting the oracle with evaluation noise---selection ranks by
$\varphi(x)+\mathcal{N}(0,\sigma^2)$, $\sigma=4$, a stand-in for a noisy
scoring function such as docking.  Now a single noisy pool gives a
high-variance estimate of $\theta$, and aggregating a few recent pools
\emph{reduces variance} faster than it adds staleness bias.
Figure~\ref{fig:optdepth} shows the resulting U-shaped curve: the expected
hitting time falls from $61.2$ rounds at $m=1$ to a minimum of $41.7$
rounds at $m^\star=5$, then rises again as stale bias dominates
($195.8$ rounds at full memory, off scale).  This is precisely the
diminishing-returns and optimal-depth structure anticipated by
Proposition~\ref{prop:sojourn-sat} and the optimal-depth question: a finite
window captures the variance-reduction benefit of memory while limiting
the bias from outdated pools.  The control strategy of
Section~\ref{sec:control} returns to this same noisy oracle and shows how
confidence-bounded, replicated decisions still recover a true
$\varepsilon$-molecule.

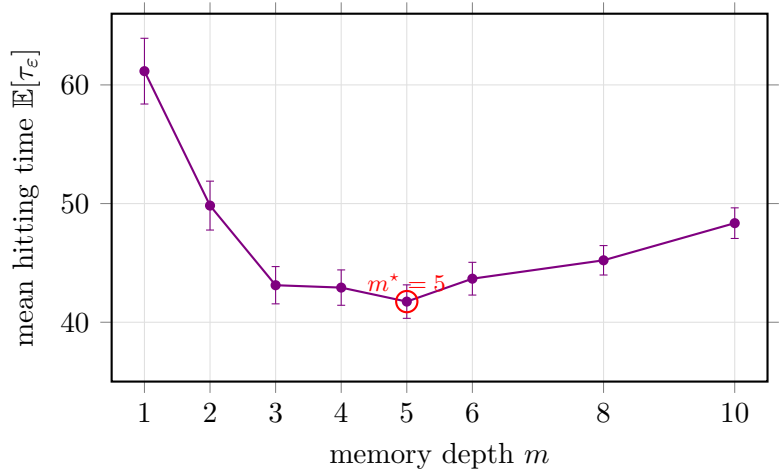
\begin{figure}[ht]
\centering
\begin{tikzpicture}
\begin{axis}[
    width=0.7\textwidth, height=0.44\textwidth,
    xlabel={memory depth $m$}, ylabel={mean hitting time $\E[\tau_\varepsilon]$},
    xmin=0.5, xmax=10.5, ymin=35, ymax=66,
    xtick={1,2,3,4,5,6,8,10},
    grid=both, major grid style={black!12},
    tick align=outside, thick,
]
\addplot[violet, mark=*, mark size=1.5pt,
    error bars/.cd, y dir=both, y explicit]
  coordinates {
    (1,61.16) +- (0,2.77) (2,49.83) +- (0,2.06) (3,43.12) +- (0,1.57)
    (4,42.92) +- (0,1.49) (5,41.74) +- (0,1.41) (6,43.67) +- (0,1.38)
    (8,45.22) +- (0,1.24) (10,48.35) +- (0,1.29)};
\addplot[only marks, mark=o, mark size=4pt, red, thick, forget plot]
  coordinates {(5,41.74)};
\node[anchor=south, red, font=\footnotesize] at (axis cs:5,41.74) {$m^\star=5$};
\end{axis}
\end{tikzpicture}
\caption{Optimal memory depth under a noisy oracle
  ($n=20$, $N=6$, $M=12$, $\sigma=4$; mean of
  $200$ runs, error bars $\pm1$ s.e.m.).  When a single pool is an
  unreliable summary, moderate memory reduces estimation variance and
  speeds convergence; excessive memory reintroduces stale bias (full
  memory: $195.8$ rounds, off scale).  The minimum at $m^\star=5$
  illustrates the diminishing-returns structure of
  Proposition~\ref{prop:sojourn-sat} and the optimal-depth question.  Because the
  plain windowed mean violates monotone learning (Assumption~\ref{ass:mono}) on
  \textsc{OneMax}, the curve is U-shaped rather than monotone.}
\label{fig:optdepth}
\end{figure}

\paragraph{The monotone hierarchy (numerical).}

When the learner \emph{does} satisfy Assumption~\ref{ass:mono}---so that
additional memory cannot lower an escape probability---Theorem~\ref{thm:mono}
and Theorem~\ref{thm:hierarchy} predict a strictly monotone benefit of
depth.  Figure~\ref{fig:hierarchy} plots the expected per-level sojourn
$1/p_k^m$ under the exponential saturation model of
Definition~\ref{def:sat}, with the parameters of the worked Example in
Section~\ref{sec:example} ($M=100$, $q_k^0=10^{-3}$, $q_k^\infty=10^{-2}$,
learning rate $c_k=0.5$).  The sojourn falls monotonically from $10.5$
rounds at $m=0$ to the saturated value $1/p_k^\infty\approx1.58$ rounds,
with exponentially diminishing marginal returns---the curve is essentially
flat beyond $m\approx6$, quantifying how a modest window already realises
almost the entire full-memory speed-up of Theorem~\ref{thm:limit}.

\begin{figure}[ht]
\centering
\begin{tikzpicture}
\begin{axis}[
    width=0.7\textwidth, height=0.44\textwidth,
    xlabel={memory depth $m$},
    ylabel={expected sojourn $1/p_k^m$ (rounds)},
    xmin=0, xmax=12, ymin=0, ymax=11,
    grid=both, major grid style={black!12},
    tick align=outside, thick,
]
\addplot[blue, mark=*, mark size=1.3pt] coordinates {
(0,10.503) (1,2.735) (2,2.045) (3,1.812) (4,1.706) (5,1.651) (6,1.621)
(7,1.603) (8,1.593) (9,1.587) (10,1.583) (11,1.581) (12,1.579)};
\addplot[black, dashed, thin, forget plot] coordinates {(0,1.577) (12,1.577)};
\node[anchor=south west, font=\footnotesize] at (axis cs:6.2,1.6)
  {saturated $1/p_k^\infty\approx1.58$};
\end{axis}
\end{tikzpicture}
\caption{Monotone memory hierarchy under monotone learning
  (Assumption~\ref{ass:mono}), saturation model \eqref{eq:sat} with $M=100$,
  $q_k^0=10^{-3}$, $q_k^\infty=10^{-2}$, $c_k=0.5$.  The expected sojourn
  $1/p_k^m$ decreases monotonically in $m$ toward the full-memory value
  (Theorems~\ref{thm:mono}, \ref{thm:hierarchy}, and~\ref{thm:limit}), with exponentially
  diminishing returns (Proposition~\ref{prop:sojourn-sat}): a window of
  $m\approx6$ already captures nearly all of the achievable
  acceleration.}
\label{fig:hierarchy}
\end{figure}
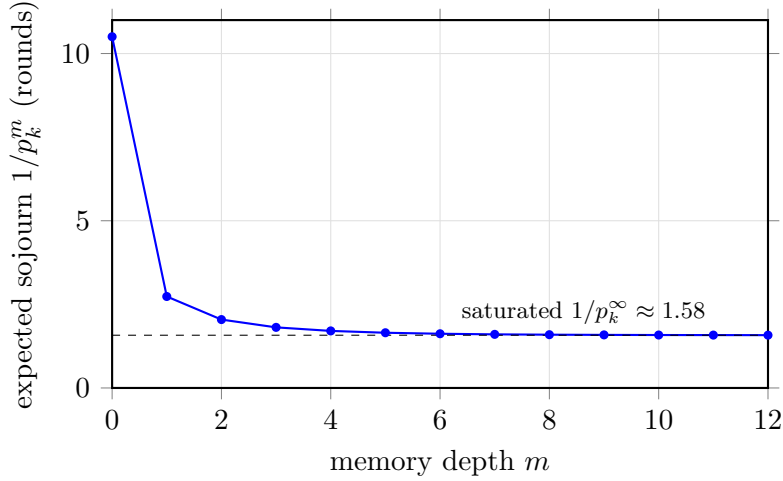

\section{Exit Times Without the Monotone-Learning Condition}
\label{sec:nonmonotone}

The memory hierarchy of Section~\ref{sec:comparison} rests on the
monotone-learning condition (Assumption~\ref{ass:mono}).  That condition, however,
was used \emph{only} to order the hitting-time bounds across memory
depths: the layered exit-time bound itself (Theorem~\ref{thm:m-htb})
requires just elitism and uniform positivity.  In practice monotone
learning can fail---under elitism the elite fitness is non-decreasing, so
older pools in a sliding window are systematically \emph{staler} (drawn
from lower levels), and a learner that averages them can be biased toward
worse proposals.  This is exactly the bias--variance trade-off observed
numerically in Section~\ref{sec:num-setup} (Figure~\ref{fig:optdepth}),
where the expected exit time is \emph{U-shaped} in the memory depth $m$
and is minimised at a finite $m^\star$.

This section builds a quantitative exit-time theory that drops
Assumption~\ref{ass:mono} entirely.  The idea is to resolve the learning
dynamics \emph{within} a fitness level: the escape probability during a
sojourn at level $k$ is governed not by a single worst-case constant but
by how many genuinely informative (current-level) pools the window
currently holds.  This yields an exit-time bound that (a) holds with no
monotonicity assumption, (b) is never weaker---and is generically
tighter---than Theorem~\ref{thm:m-htb}, and (c) pinpoints the optimal
memory depth as the first peak of a one-dimensional ``learning profile.''
Throughout we assume only Assumption~\ref{ass:elitism} (elitism) and
Assumption~\ref{ass:m-pos} ($m$-step uniform positivity); the window
process is the time-homogeneous Markov chain of Theorem~\ref{thm:m-markov}.

\subsection{The level-\texorpdfstring{$k$}{k} learning profile and the non-monotone bound}
\label{sec:nm-profile}

During a sojourn at level $k$, elitism forces a rigid window structure.
If the chain enters level $k$ at time $\tau_k$ and has spent $s$ rounds
there, then $P_{\tau_k},\dots,P_{\tau_k+s}$ are all at level $k$
while every earlier pool is at level $\le k$; hence the window
$S^m_{\tau_k+s}=(P_{\tau_k+s-m+1},\dots,P_{\tau_k+s})$ has its
$\min(s+1,m)$ most recent entries at level $k$ and the remainder at
levels $\le k$.  We index escape probabilities by this count.

\begin{definition}[Level-$k$ learning profile]\label{def:profile}
  For $1\le j\le m$ let
  \[
    \mathcal W_k(j)=\bigl\{\,S=(Q_1,\dots,Q_m)\in\Omn^m :
      \ell(Q_i)=k\ (m-j<i\le m),\quad \ell(Q_i)\le k\ (i\le m-j)\,\bigr\}
  \]
  be the windows whose $j$ most recent pools are at level $k$ and whose
  remaining $m-j$ pools are stale (at levels $\le k$).  Define the
  \emph{level-$k$ learning profile} and its $M$-fold lift
  \[
    \rho_k(j)=\inf_{S\in\mathcal W_k(j)} G(S)\bigl(\Akp\mid Q_m\bigr),
    \qquad
    r_k(j)=1-\bigl(1-\rho_k(j)\bigr)^M ,\qquad j=1,\dots,m .
  \]
\end{definition}

The sets $\mathcal W_k(j)$ are precisely the window compositions that
\emph{occur} during a level-$k$ sojourn; the profile thus measures the
escape probability as a function of the amount of current-level data the
learner has accumulated.  Under Assumption~\ref{ass:m-pos},
$\rho_k(j)\ge\alpha\,|\Akp|>0$ for every $j$, so each $r_k(j)\in(0,1]$.

\begin{remark}[Relation to the worst-case escape probability]
\label{rem:profile-vs-qkm}
  The escape probability $q_k^m$ of Definition~\ref{def:m-qk} is the
  infimum over \emph{all} windows with $\ell(Q_m)=k$, which is a superset
  of $\bigcup_j\mathcal W_k(j)$ and additionally contains windows with an
  entry \emph{above} level $k$---configurations elitism renders
  unreachable during a level-$k$ sojourn.  Consequently
  $q_k^m\le\min_{1\le j\le m}\rho_k(j)$, i.e.\ the profile is the
  reachable, level-resolved refinement of $q_k^m$.
\end{remark}

\paragraph{The non-monotone exit-time bound.}

\begin{lemma}[Inhomogeneous sojourn domination]\label{lem:nm-sojourn}
  Define $\Pi_k(s)=\prod_{j=1}^{s}\bigl(1-r_k(j)\bigr)$ with
  $\Pi_k(0)=1$, and
  \begin{equation}
    T_k(m)=\sum_{s\ge0}\ \prod_{i=0}^{s-1}\Bigl(1-r_k\bigl(\min(i+1,m)\bigr)\Bigr)
          =\sum_{s=0}^{m-2}\Pi_k(s)\;+\;\frac{\Pi_k(m-1)}{r_k(m)} .
    \label{eq:Tkm}
  \end{equation}
  Then under Assumptions~\ref{ass:elitism} and~\ref{ass:m-pos} the
  expected sojourn at level $k$ satisfies $\E[\tau_{k+1}-\tau_k]\le
  T_k(m)$.
\end{lemma}

\begin{proof}
  By the window structure above, at the $s$-th round of the sojourn
  $S^m_{\tau_k+s}\in\mathcal W_k(\min(s+1,m))$, so the single-proposal
  escape probability is at least $\rho_k(\min(s+1,m))$ and, with $M$
  i.i.d.\ proposals, the round escape probability conditional on the
  past is at least $r_k(\min(s+1,m))$; a single escaping proposal raises
  the level by elitism.  Writing $N_k=\tau_{k+1}-\tau_k$, the same
  conditioning argument as in Lemma~\ref{lem:sojourn} gives
  $\Prob(N_k>s)\le\prod_{i=0}^{s-1}\bigl(1-r_k(\min(i+1,m))\bigr)$, whence
  $\E[N_k]=\sum_{s\ge0}\Prob(N_k>s)\le T_k(m)$.  For the closed form in
  \eqref{eq:Tkm}: for $s\le m-1$ the survival product equals
  $\Pi_k(s)$, while for $s\ge m-1$ it equals
  $\Pi_k(m-1)\,(1-r_k(m))^{\,s-(m-1)}$; summing the geometric tail yields
  $\Pi_k(m-1)/r_k(m)$.
\end{proof}

\begin{theorem}[Non-monotone layered exit time]\label{thm:nm-exit}
  Under Assumptions~\ref{ass:elitism} and~\ref{ass:m-pos}---and
  \emph{without} Assumption~\ref{ass:mono}---the $m$-step algorithm
  satisfies, for any $P_0$ with $\ell(P_0)=k_0$,
  \begin{equation}
    \E[\tau_\varepsilon\mid P_0]\ \le\ \sum_{k=k_0}^{K(\varepsilon)-1} T_k(m).
    \label{eq:nm-exit}
  \end{equation}
  Moreover $T_k(m)\le 1/p_k^m$ for every $k$, so \eqref{eq:nm-exit} is
  never weaker than the layered bound of Theorem~\ref{thm:m-htb}, and is
  strictly tighter whenever the profile $\rho_k(\cdot)$ is non-constant.
\end{theorem}

\begin{proof}
  By Lemma~\ref{lem:monotone} the levels are non-decreasing, so
  $\tau_\varepsilon=\tau_K=\sum_{k=k_0}^{K(\varepsilon)-1}(\tau_{k+1}-\tau_k)$
  ($P_0$-a.s.\ finite, since Assumption~\ref{ass:m-pos} gives the geometric
  tail exactly as in Theorem~\ref{thm:m-htb}).  Taking expectations and
  applying Lemma~\ref{lem:nm-sojourn} term by term gives \eqref{eq:nm-exit}.
  For the comparison, Remark~\ref{rem:profile-vs-qkm} gives $r_k(j)\ge
  p_k^m$ for all $j$, so every survival factor in \eqref{eq:Tkm} is at
  most $1-p_k^m$, whence $T_k(m)\le\sum_{s\ge0}(1-p_k^m)^s=1/p_k^m$.
\end{proof}

\begin{proposition}[Exactness under composition-homogeneity]
\label{prop:nm-exact}
  If the learner is \emph{composition-homogeneous}---the escape
  probability equals $\rho_k(j)$ for every window in $\mathcal W_k(j)$,
  not merely as an infimum---and levels are entered without skipping,
  then $\E[\tau_{k+1}-\tau_k]=T_k(m)$ and
  $\E[\tau_\varepsilon\mid P_0]=\sum_{k=k_0}^{K(\varepsilon)-1}T_k(m)$.
\end{proposition}

\begin{proof}
  Under homogeneity the inequalities in Lemma~\ref{lem:nm-sojourn} become
  equalities, so $N_k$ is the inhomogeneous geometric variable with
  survival $\prod_{i<s}(1-r_k(\min(i+1,m)))$ and mean exactly $T_k(m)$;
  with no skipping every level is visited and the sojourns sum to
  $\tau_\varepsilon$.
\end{proof}

\subsection{The optimal memory depth}
\label{sec:nm-optdepth}

We now compare the depth-$m$ and depth-$(m{+}1)$ algorithms.  Both use
the \emph{same} learner---hence the same profile values
$\rho_k(1),\rho_k(2),\dots$---and differ only in where the window
saturates.

\begin{theorem}[Marginal value of memory; optimal depth]\label{thm:nm-margin}
  For every level $k$ and every $m\ge1$,
  \begin{equation}
    T_k(m+1)-T_k(m)
    =\Pi_k(m-1)\left(1-\frac{1}{r_k(m)}+\frac{1-r_k(m)}{r_k(m+1)}\right),
    \label{eq:nm-margin}
  \end{equation}
  and, since $\Pi_k(m-1)>0$,
  \begin{equation}
    \operatorname{sign}\!\bigl(T_k(m+1)-T_k(m)\bigr)
    =-\operatorname{sign}\!\bigl(\rho_k(m+1)-\rho_k(m)\bigr).
    \label{eq:nm-signlaw}
  \end{equation}
  Thus $T_k(\cdot)$ strictly decreases while the profile rises and
  strictly increases once it falls.  The exit-time-minimising level-$k$
  depth is the \emph{first peak of the learning profile},
  \begin{equation}
    m_k^\star=\min\{\,m\ge1 : \rho_k(m+1)\le\rho_k(m)\,\},
    \label{eq:nm-mstar}
  \end{equation}
  and if $\rho_k(\cdot)$ is unimodal then $m_k^\star$ is its global
  minimiser.
\end{theorem}

\begin{proof}
  From \eqref{eq:Tkm},
  $T_k(m)=\sum_{s=0}^{m-2}\Pi_k(s)+\Pi_k(m-1)/r_k(m)$, so
  \[
    T_k(m+1)-T_k(m)
    =\Pi_k(m-1)+\frac{\Pi_k(m)}{r_k(m+1)}-\frac{\Pi_k(m-1)}{r_k(m)} .
  \]
  Using $\Pi_k(m)=\Pi_k(m-1)\bigl(1-r_k(m)\bigr)$ and factoring
  $\Pi_k(m-1)$ gives \eqref{eq:nm-margin}.  Writing $b=r_k(m)$,
  $b'=r_k(m+1)$, the bracket is
  $1-\tfrac1b+\tfrac{1-b}{b'}$, which is $<0$ iff
  $\tfrac{1-b}{b'}<\tfrac{1-b}{b}$ iff (as $0<b<1$) $b'>b$, i.e.\ iff
  $\rho_k(m+1)>\rho_k(m)$ (since $p\mapsto1-(1-p)^M$ is increasing); this
  is \eqref{eq:nm-signlaw}.  Hence $T_k$ decreases exactly up to the
  first index where the profile stops increasing, which is
  \eqref{eq:nm-mstar}; unimodality makes this the unique global
  minimiser.
\end{proof}

\begin{corollary}[Recovery of the monotone hierarchy]\label{cor:nm-recover}
  If $\rho_k(\cdot)$ is non-decreasing for every $k$---the natural
  restatement of monotone learning (Assumption~\ref{ass:mono})---then by
  \eqref{eq:nm-signlaw} each $T_k(\cdot)$ is non-increasing, $m_k^\star=\infty$,
  and ``more memory never hurts,'' recovering Theorem~\ref{thm:hierarchy}.  At
  the other extreme $m=1$ gives $T_k(1)=1/r_k(1)$, the single-step bound
  with $q_k=\rho_k(1)$.
\end{corollary}

\paragraph{A bias--variance profile model.}

The saturation model of Definition~\ref{def:sat} is the monotone special
case $\rho_k(j)=q_k^\infty-(q_k^\infty-q_k^0)e^{-c_kj}$.  To capture the
non-monotone regime we add an explicit staleness penalty: with one more
remembered pool the learner gains information (a saturating learning
term) but also conditions on data one round older (a linear staleness
term).

\begin{proposition}[Bias--variance profile]\label{prop:nm-biasvar}
  Let
  \begin{equation}
    \rho_k(j)=q_k^0+\underbrace{g_k\bigl(1-e^{-c_kj}\bigr)}_{\text{learning gain}}
              -\underbrace{d_k\,(j-1)}_{\text{staleness bias}},
    \qquad g_k,c_k>0,\ d_k\ge0,
    \label{eq:biasvar}
  \end{equation}
  clamped to $(0,1)$.  Then:
  \begin{enumerate}[label=(\alph*),nosep]
  \item if $0<d_k<g_kc_k$, the profile is unimodal with continuous peak
    $j_k^\star=c_k^{-1}\ln\!\bigl(g_kc_k/d_k\bigr)$, so
    $m_k^\star\in\{\lfloor j_k^\star\rfloor,\lceil j_k^\star\rceil\}$;
  \item if $d_k\ge g_kc_k$, then $\rho_k$ is non-increasing and
    $m_k^\star=1$ (memory never helps at level $k$);
  \item if $d_k=0$, \eqref{eq:biasvar} reduces to the saturation model
    and $m_k^\star=\infty$ (the monotone hierarchy).
  \end{enumerate}
\end{proposition}

\begin{proof}
  Treating $j$ as continuous, $\rho_k'(j)=g_kc_ke^{-c_kj}-d_k$ and
  $\rho_k''(j)=-g_kc_k^2e^{-c_kj}<0$, so $\rho_k$ is concave; setting
  $\rho_k'=0$ gives $j_k^\star=c_k^{-1}\ln(g_kc_k/d_k)$, which lies in
  $(0,\infty)$ iff $0<d_k<g_kc_k$.  The boundary cases follow by
  inspection, and \eqref{eq:nm-mstar} converts the peak into the integer
  optimum.
\end{proof}

\paragraph{Illustration.}
Table~\ref{tab:nm-example} tabulates \eqref{eq:Tkm} for $M=8$ and the profile
\eqref{eq:biasvar} with $q_k^0=0.006$, $g_k=0.05$, $c_k=0.55$,
$d_k=0.005$ (so $g_kc_k=0.0275>d_k$, giving an interior optimum).  The
sojourn $T_k(m)$ is U-shaped, minimised at $m_k^\star=3$ in exact
agreement with the first-peak rule \eqref{eq:nm-mstar}; deeper memory
then degrades sharply as staleness dominates, mirroring the empirical
U-shape of Figure~\ref{fig:optdepth}.

\begin{table}[ht]
\centering
\caption{Non-monotone sojourn $T_k(m)$ from \eqref{eq:Tkm} under the
  bias--variance profile \eqref{eq:biasvar} ($M=8$).  The profile peaks
  at $j=3$, and $T_k$ is minimised there ($m_k^\star=3$), as predicted by
  Theorem~\ref{thm:nm-margin}.}
\label{tab:nm-example}
\renewcommand{\arraystretch}{1.2}
\begin{tabular}{@{}rccc@{}}
\toprule
$m$ & $\rho_k(m)$ & $r_k(m)$ & $T_k(m)$ \\
\midrule
1  & 0.0272 & 0.198 & 5.06 \\
2  & 0.0344 & 0.244 & 4.29 \\
3  & 0.0364 & 0.257 & \textbf{4.17} \\
4  & 0.0355 & 0.251 & 4.21 \\
5  & 0.0328 & 0.234 & 4.30 \\
6  & 0.0292 & 0.211 & 4.43 \\
8  & 0.0204 & 0.152 & 4.76 \\
10 & 0.0108 & 0.083 & 5.47 \\
12 & 0.0009 & 0.007 & 18.88 \\
\bottomrule
\end{tabular}
\end{table}

\paragraph{Level-adaptive memory.}

Because the running level $\ell(P_t)$ is observable---it is the rank of
the current best fitness---the algorithm can choose its window depth as
a function of the level.  Using depth $m_k^\star$ while at level $k$
leaves each per-level sojourn analysis unchanged (it depends only on the
within-level window evolution), so the resulting exit-time bound is
$\sum_k T_k(m_k^\star)=\sum_k\min_m T_k(m)$.

\begin{corollary}[Adaptive depth dominates any fixed depth]
\label{cor:nm-adaptive}
  Let $m^\star=\arg\min_m\sum_k T_k(m)$ be the best fixed depth.  The
  level-adaptive policy that uses depth $m_k^\star$ at level $k$ satisfies
  \[
    \sum_{k=k_0}^{K(\varepsilon)-1}T_k(m_k^\star)
    =\sum_{k=k_0}^{K(\varepsilon)-1}\min_{m}T_k(m)
    \ \le\ \min_{m}\sum_{k=k_0}^{K(\varepsilon)-1}T_k(m)
    =\sum_{k=k_0}^{K(\varepsilon)-1}T_k(m^\star),
  \]
  with strict inequality whenever the optimal depths $m_k^\star$ are not
  all equal.
\end{corollary}

\begin{proof}
  Termwise, $T_k(m_k^\star)=\min_m T_k(m)\le T_k(m^\star)$; sum over $k$.
  Strictness holds as soon as some $T_k(m^\star)>\min_m T_k(m)$, i.e.\
  $m_k^\star\ne m^\star$ for some $k$.
\end{proof}

Thus, in the absence of monotone learning, the right design is not
``use as much memory as possible'' (Theorem~\ref{thm:hierarchy}) but
``match the memory depth to the learning profile of the current level.''
When the profile is increasing this prescribes unbounded memory and
recovers the hierarchy; when staleness bites it prescribes a finite,
level-tuned window, and Corollary~\ref{cor:nm-adaptive} quantifies the
gain over any one-size-fits-all depth.

\subsection{Numerical illustration: within-level dynamics}
\label{sec:num-within}

The global U-shape of Figure~\ref{fig:optdepth} is the aggregate of what
happens at each fitness level.  Section~\ref{sec:nonmonotone} predicts
that, when monotone learning fails, the governing object is the level-$k$
learning profile $\rho_k(j)$---the escape probability as a function of
the number $j$ of current-level pools the window holds---and that the
per-level sojourn is U-shaped with a finite optimal depth
(Theorems~\ref{thm:nm-margin} and~\ref{prop:nm-biasvar}).  We measure both
directly on the noisy-oracle algorithm of Section~\ref{sec:num-memory}.

Figure~\ref{fig:nm-profile} shows the empirical profile at three levels.
Each is \emph{unimodal}---rising as additional current-level pools reduce
estimation variance, then falling as the oldest pooled data turn
stale---exactly the bias--variance shape posited in
Proposition~\ref{prop:nm-biasvar}.  Moreover the peak shifts to larger
$j$ as the level $k$ increases: harder levels reward deeper memory.

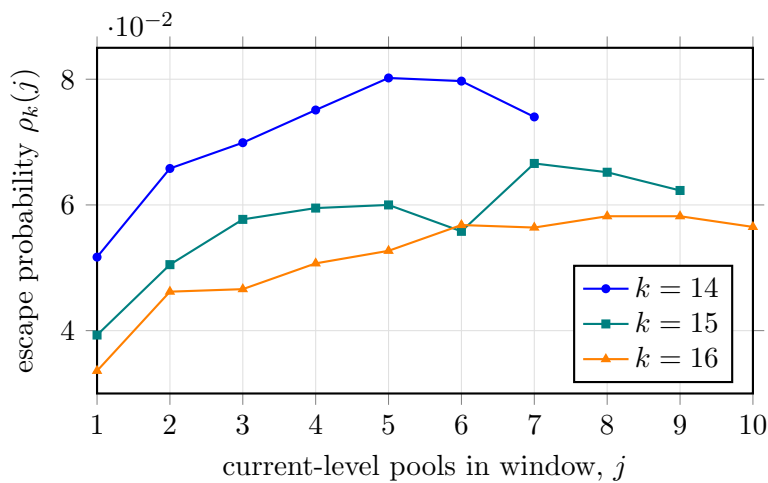
\begin{figure}[ht]
\centering
\begin{tikzpicture}
\begin{axis}[
    width=0.7\textwidth, height=0.42\textwidth,
    xlabel={current-level pools in window, $j$},
    ylabel={escape probability $\rho_k(j)$},
    xmin=1, xmax=10, ymin=0.03, ymax=0.085,
    legend pos=south east, legend cell align=left,
    grid=both, major grid style={black!12}, tick align=outside, thick,
]
\addplot[blue, mark=*, mark size=1.3pt] coordinates {
(1,0.0517)(2,0.0658)(3,0.0699)(4,0.0751)(5,0.0802)(6,0.0797)(7,0.0740)};
\addlegendentry{$k=14$}
\addplot[teal, mark=square*, mark size=1.3pt] coordinates {
(1,0.0393)(2,0.0505)(3,0.0577)(4,0.0595)(5,0.0600)(6,0.0558)(7,0.0666)(8,0.0652)(9,0.0623)};
\addlegendentry{$k=15$}
\addplot[orange, mark=triangle*, mark size=1.5pt] coordinates {
(1,0.0336)(2,0.0462)(3,0.0466)(4,0.0507)(5,0.0527)(6,0.0568)(7,0.0564)(8,0.0582)(9,0.0582)(10,0.0565)};
\addlegendentry{$k=16$}
\end{axis}
\end{tikzpicture}
\caption{Empirical level-$k$ learning profile $\rho_k(j)$
  (Definition~\ref{def:profile}) on noisy \textsc{OneMax}
  ($N=6$, $M=12$, $\sigma=4$; $\ge10^3$ proposals per point, $900$ runs).
  Each profile is unimodal, and its peak moves to larger $j$ as the level
  rises---the empirical counterpart of the bias--variance profile of
  Proposition~\ref{prop:nm-biasvar}.}
\label{fig:nm-profile}
\end{figure}

Figure~\ref{fig:nm-sojourn} confirms the consequence for exit time: the
measured per-level sojourn $\bar T_k(m)$ is U-shaped in the memory depth
$m$, with a finite minimiser $m_k^\star$ that \emph{increases with the
level},
\[
  m_k^\star = 1,\,2,\,2,\,3,\,4,\,8,\,9 \quad\text{for}\quad k=13,\dots,19 .
\]
The global optimum $m^\star=5$ of Figure~\ref{fig:optdepth} is therefore a
cross-level compromise---too shallow for the high bottleneck levels, too
deep for the low ones---which is exactly why no single fixed depth is
ideal and why Corollary~\ref{cor:nm-adaptive} proposes matching the depth
to the level.

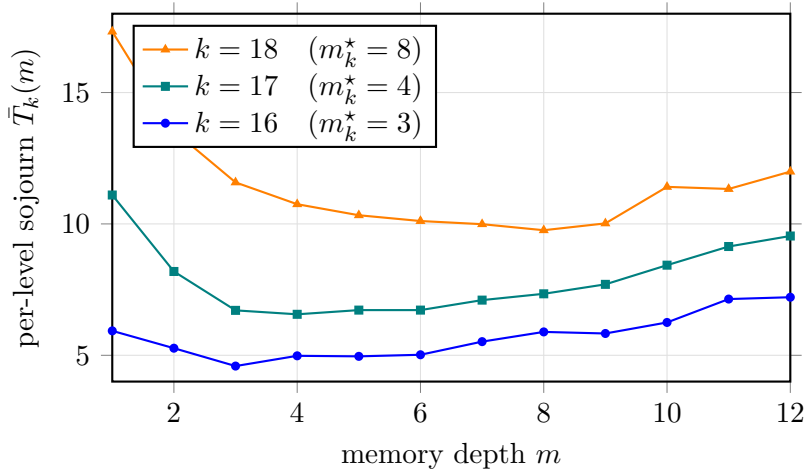
\begin{figure}[ht]
\centering
\begin{tikzpicture}
\begin{axis}[
    width=0.72\textwidth, height=0.44\textwidth,
    xlabel={memory depth $m$}, ylabel={per-level sojourn $\bar T_k(m)$},
    xmin=1, xmax=12, ymin=4, ymax=18,
    legend pos=north west, legend cell align=left,
    grid=both, major grid style={black!12}, tick align=outside, thick,
]
\addplot[orange, mark=triangle*, mark size=1.4pt] coordinates {
(1,17.32)(2,13.54)(3,11.58)(4,10.75)(5,10.33)(6,10.11)(7,9.99)(8,9.76)(9,10.02)(10,11.41)(11,11.33)(12,11.99)};
\addlegendentry{$k=18$\quad($m_k^\star=8$)}
\addplot[teal, mark=square*, mark size=1.3pt] coordinates {
(1,11.10)(2,8.19)(3,6.71)(4,6.56)(5,6.72)(6,6.72)(7,7.10)(8,7.34)(9,7.70)(10,8.43)(11,9.14)(12,9.54)};
\addlegendentry{$k=17$\quad($m_k^\star=4$)}
\addplot[blue, mark=*, mark size=1.3pt] coordinates {
(1,5.93)(2,5.27)(3,4.59)(4,4.98)(5,4.96)(6,5.02)(7,5.52)(8,5.89)(9,5.83)(10,6.25)(11,7.14)(12,7.21)};
\addlegendentry{$k=16$\quad($m_k^\star=3$)}
\end{axis}
\end{tikzpicture}
\caption{Measured per-level sojourn $\bar T_k(m)$ versus memory depth
  (same setting as Figure~\ref{fig:nm-profile}; $500$ runs per depth).  Every
  level is U-shaped---memory helps then hurts---and the optimal depth
  $m_k^\star$ grows with the level, confirming the finite, level-dependent
  optimum predicted by Theorem~\ref{thm:nm-margin}.}
\label{fig:nm-sojourn}
\end{figure}

\paragraph{Scope of the sharp results.}
The first-peak identity (Theorem~\ref{thm:nm-margin}) and the closed-form
sojourn (Proposition~\ref{prop:nm-exact}) hold \emph{exactly} under
composition-homogeneity---when the escape probability depends only on the
count $j$ of current-level pools.  The naive windowed learner used here
additionally pools stale \emph{lower}-level data, which couples the
levels and makes it only approximately composition-homogeneous;
accordingly it confirms the theory qualitatively (unimodal profile,
finite level-dependent optimum) but not to the digit, and a
level-adaptive schedule built from the measured $m_k^\star$ does not beat
the best fixed depth in our runs.  A level-gated learner (training only
on current-level pools) restores approximate composition-homogeneity but,
by removing the staleness bias, also flattens the non-monotonicity.
Designing a single learner that is at once strongly non-monotone and
composition-homogeneous---so that the level-adaptive gain of
Corollary~\ref{cor:nm-adaptive} is realised---is a concrete algorithmic
question raised by the theory.  This mirrors the status of monotone
learning in Section~\ref{sec:comparison}: the numerics confirm the
qualitative theory, while the sharp quantitative statements isolate
exactly the extra structure they require.

\section{The Cost of Evaluations}
\label{sec:control}

The bounds so far count \emph{rounds}.  In practice the binding resource
is the number of \emph{fitness-oracle evaluations}---docking runs, assays,
syntheses.  We now recast the theory in the currency of evaluations and
identify the evaluation-optimal operation of the algorithm; the treatment
of noisy oracles follows in Section~\ref{sec:ctrl-noise}.  Its experiments
run on the discrete test problem of Section~\ref{sec:num-setup}.

\subsection{Evaluation cost and the optimal budget}
\label{sec:ctrl-budget}

\begin{definition}[Evaluation cost]\label{def:evalcost}
  Let $\evals_\varepsilon$ be the number of oracle evaluations of
  $\varphi$ until an $\varepsilon$-molecule is first scored.  For
  \Cref{alg:gen-sel} ($N$ initial scorings, then $M$ proposals per round),
  $\evals_\varepsilon=N+M\,\tau_\varepsilon$, so by
  Theorem~\ref{thm:base-htb}
  $\E[\evals_\varepsilon\mid P_0]\le N+M\sum_{k=k_0}^{K(\varepsilon)-1}1/p_k$.
\end{definition}

Escaping level $k$ costs, in expectation, $M/p_k$ evaluations
($1/p_k$ rounds of $M$ scorings each), i.e.\
$E_k(M)=M/\bigl(1-(1-q_k)^M\bigr)$.

\begin{theorem}[Monotone evaluation cost]\label{thm:monoeval}
  For every $q\in(0,1)$, $M\mapsto E(M)=M/\bigl(1-(1-q)^M\bigr)$ is
  non-decreasing on integers $M\ge1$; hence $\min_{M\ge1}E(M)=E(1)=1/q$.
\end{theorem}

\begin{proof}
  With $a=1-q$, $E(M{+}1)\ge E(M)$ reduces to $a^{M}(1+Mq)\le1$.  Since
  $a^{M}=(1-q)^M\le e^{-Mq}$ and $1+Mq\le e^{Mq}$, the product is $\le1$.
\end{proof}

\begin{corollary}[Evaluation-minimal budget]\label{cor:minbudget}
  $\E[\evals_\varepsilon\mid P_0]\le N+\sum_{k=k_0}^{K(\varepsilon)-1}M/(1-(1-q_k)^M)$
  is non-decreasing in $M$ and minimised at $M=1$:
  \[
    \E[\evals_\varepsilon\mid P_0]\ \le\ N+\sum_{k=k_0}^{K(\varepsilon)-1}\frac1{q_k}
    \ =\ N+D,
  \]
  the landscape difficulty (Definition~\ref{def:difficulty}).  Among all batch
  sizes, sequential evaluation $(M=1)$ minimises oracle calls: a batch of
  $M$ is scored in full even when its first proposal already escapes, and
  the surplus $M-1$ scorings are wasted.
\end{corollary}

This is a deceptively simple but practically important result.  In drug
design, where each oracle call---a docking run, an assay, a synthesis and
measurement---dominates the cost, the optimal strategy is to propose one
candidate at a time rather than large batches.  The intuition is
straightforward: when a batch of $M$ candidates is scored, even if the
very first one already achieves the target fitness, the remaining $M-1$
evaluations are wasted.

\begin{remark}[Evaluation--round Pareto frontier]\label{rem:pareto}
  The budget governs two opposing costs,
  $\E[\evals_\varepsilon]\le N+M\sum_k1/p_k$ (evaluations, $\uparrow$ in
  $M$) versus $\E[\tau_\varepsilon]\le\sum_k1/p_k$ (rounds, $\downarrow$ in
  $M$).  Large $M$ buys fewer rounds (wall-clock, if the oracle is
  parallelisable) at the price of more evaluations; the extremes are $M=1$
  (evaluation-optimal) and $M\to\infty$ (one round per level).  For
  drug-design oracles, where each call dominates cost, the relevant corner
  is $M=1$.
\end{remark}

\subsection{The evaluation-minimal algorithm}
\label{sec:ctrl-algo}

Three refinements lower the constant in Corollary~\ref{cor:minbudget} without
weakening the guarantee.  \emph{(i) Memoisation:} cache scored molecules
and charge only cache misses, replacing $\evals_\varepsilon$ by the count
of \emph{distinct} molecules scored.  \emph{(ii) Early stopping:}
interleave scoring with the stopping test, halting on the first
$\varepsilon$-molecule.  \emph{(iii) Level-adaptive memory:} at $M=1$ the
per-level evaluation cost is exactly the within-level sojourn $T_k(m)$ of
Section~\ref{sec:nonmonotone} (eq.~\eqref{eq:Tkm} at $M=1$), so choosing the depth
$m_k^\star=\argmin_m T_k(m)$ of Theorem~\ref{thm:nm-margin} minimises it; since
$\ell(P_t)$ is observable this is online.  Combining the three refinements
yields EM-SGS (\Cref{alg:emsgs}).

\begin{algorithm}[ht]
\caption{Evaluation-Minimal Sequential Generative Selection (EM-SGS)}\label{alg:emsgs}
\begin{algorithmic}[1]
\Require $\varepsilon$; pool $P_0$ ($N$ scored molecules); depths $\{m_k^\star\}$
\State cache $\cache\leftarrow P_0$ with their oracle values;\ \ $\evals\leftarrow N$
\Repeat
  \State $k\leftarrow\ell(P_t)$; build $G_t$ from the last $m_k^\star$ elite pools; sample $y\sim G_t(\cdot\mid P_t)$
  \If{$y\in\cache$}
    \State reuse the cached value
  \Else
    \State score $\varphi(y)$;\ $\cache\leftarrow\cache\cup\{y\}$;\ $\evals\leftarrow\evals+1$
  \EndIf
  \If{$\varphi(y)\ge\phist-\varepsilon$}
    \Return $y,\ \evals$ \Comment{early stop}
  \EndIf
  \State $P_{t+1}\leftarrow\topN(P_t\cup\{y\})$;\ refit $G_{t+1}$
\Until{stopped}
\end{algorithmic}
\end{algorithm}

\begin{proposition}[Guarantee for EM-SGS]\label{prop:emsgs}
  Under Assumptions~\ref{ass:elitism} and~\ref{ass:m-pos}, EM-SGS reaches
  an $\varepsilon$-molecule almost surely, with
  \[
    \E[\evals_\varepsilon\mid P_0]
    \ \le\ N+\sum_{k=k_0}^{K(\varepsilon)-1}\min_m T_k(m)
    \ \le\ N+\sum_{k=k_0}^{K(\varepsilon)-1}\frac1{q_k^{m_k^\star}},
  \]
  and memoisation makes the realised distinct-evaluation count no larger.
\end{proposition}

\begin{proof}
  At $M=1$, $p_k=q_k>0$ under Assumption~\ref{ass:m-pos}, so a.s.\ convergence
  follows from Theorem~\ref{thm:base-conv} with $M=1$.  The per-level evaluation
  cost equals the $M=1$ sojourn $T_k(m)$; minimising over $m$ gives
  $\min_m T_k(m)\le T_k(m_k^\star)\le1/q_k^{m_k^\star}$ by the bound
  $T_k(m)\le1/p_k^m$ of Theorem~\ref{thm:nm-exit} at $M=1$.  Caching only removes
  repeated scorings.
\end{proof}

EM-SGS is thus the joint minimiser over the two design axes: the budget
($M=1$, Theorem~\ref{thm:monoeval}) and the memory depth ($m_k^\star$,
Theorem~\ref{thm:nm-margin}).

\paragraph{Numerical illustration: the evaluation--round frontier.}

We test on the \textsc{OneMax} space of Section~\ref{sec:num-setup}.

\paragraph{Evaluation--round frontier (noiseless).}
Table~\ref{tab:eval} and Figure~\ref{fig:pareto} report total oracle evaluations and
rounds to the optimum ($n=40$, $N=20$, $m=1$; mean of $200$ runs) as $M$
varies.  Evaluations rise with $M$ exactly as Theorem~\ref{thm:monoeval}
predicts---minimal at $M=1$ ($240$ calls), $29\%$ below $M=32$
($338$)---while rounds fall from $236$ to $11$; memoisation and early
stopping remove a further $\sim\!6\%$.  The frontier has a knee near
$M=2$--$3$ (rounds already halved for almost no extra evaluations).

\begin{table}[ht]
\centering
\caption{Evaluation cost versus budget on \textsc{OneMax}
  ($n=40$, $N=20$, $m=1$; mean of $200$ runs).}
\label{tab:eval}
\renewcommand{\arraystretch}{1.15}
\begin{tabular}{@{}rccc@{}}
\toprule
budget $M$ & evals (plain) & evals (memo.+early stop) & rounds \\
\midrule
1 & 256 & \textbf{240} & 236 \\
2 & 257 & 241 & 118 \\
5 & 261 & 248 & 48 \\
8 & 272 & 258 & 32 \\
12 & 290 & 274 & 22 \\
20 & 314 & 298 & 15 \\
32 & 359 & 338 & 11 \\
\bottomrule
\end{tabular}
\end{table}

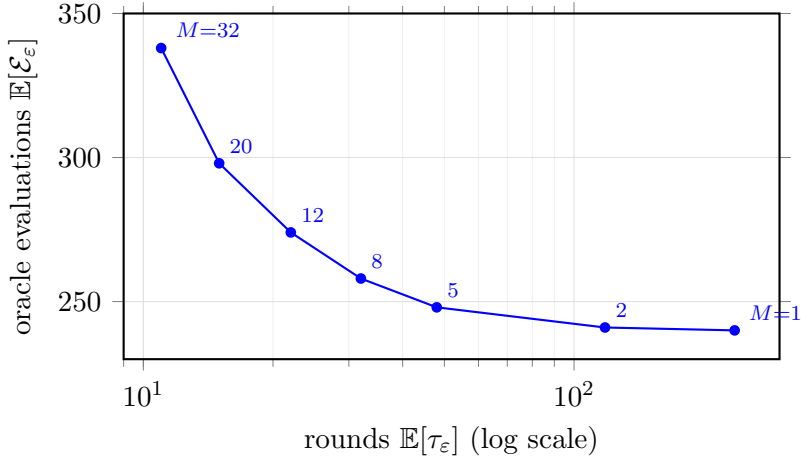
\begin{figure}[ht]
\centering
\begin{tikzpicture}
\begin{axis}[
    width=0.7\textwidth, height=0.42\textwidth,
    xlabel={rounds $\E[\tau_\varepsilon]$ (log scale)},
    ylabel={oracle evaluations $\E[\evals_\varepsilon]$},
    xmode=log, xmin=9, xmax=300, ymin=230, ymax=350,
    grid=both, major grid style={black!12}, minor grid style={black!6},
    tick align=outside, thick,
    nodes near coords, point meta=explicit symbolic,
    every node near coord/.append style={font=\scriptsize, anchor=south west},
]
\addplot[blue, mark=*, mark size=1.6pt] coordinates {
(236,240)[\,$M{=}1$] (118,241)[$2$] (48,248)[$5$] (32,258)[$8$]
(22,274)[$12$] (15,298)[$20$] (11,338)[\,$M{=}32$]};
\end{axis}
\end{tikzpicture}
\caption{Evaluation--round Pareto frontier on \textsc{OneMax}, parametrised
  by the budget $M$ (mean of $200$ runs; evaluations with memoisation and
  early stopping).  Lower-right ($M=1$) minimises oracle calls;
  upper-left ($M=32$) minimises rounds.}
\label{fig:pareto}
\end{figure}

\section{Noisy Oracles}\label{sec:noisy}

Real oracles---docking runs, assays, simulations---return $\varphi(x)$
corrupted by noise.  This section shows how the certification of an
$\varepsilon$-molecule, and its cost, change first under sub-Gaussian and
then under heavy-tailed noise, building on the evaluation-cost view above.

\subsection{Sub-Gaussian noise}
\label{sec:ctrl-noise}

Real oracles are \emph{noisy}: a query returns
$\tilde\varphi(x)=\varphi(x)+\xi$ with $\xi$ zero-mean,
$\sigma^2$-sub-Gaussian, independent across queries---the same noisy
oracle used in the numerical study of Section~\ref{sec:num-memory}.  Noise
undermines all three ingredients of Section~\ref{sec:ctrl-algo}: \emph{elitism}
on noisy scores may keep lucky molecules and drop good ones, so
$\ell(P_t)$ need no longer be monotone; \emph{memoisation} freezes a noisy
value; and \emph{stopping} fires falsely when
$\xi\ge\phist-\varepsilon-\varphi(x)$ pushes a sub-optimal molecule past
the threshold.

The remedy is to spend evaluations on \emph{certainty}.  Averaging $r$
queries, $\widehat\varphi_r(x)=\tfrac1r\sum_{i\le r}\tilde\varphi_i(x)$,
concentrates,
$\Prob(|\widehat\varphi_r(x)-\varphi(x)|\ge c)\le2e^{-rc^2/2\sigma^2}$,
with anytime radius $c_r(\delta)=\sigma\sqrt{2\ln(2/\delta)/r}$ and bounds
$\mathrm{LCB}_r=\widehat\varphi_r-c_r$,
$\mathrm{UCB}_r=\widehat\varphi_r+c_r$.  Each test is replaced by its
confident form: prefer $x$ to $x'$ only when
$\mathrm{LCB}(x)>\mathrm{UCB}(x')$; let the cache hold the running
$(\text{sum},\text{count})$ so re-querying \emph{refines} (memoisation
becomes \emph{accumulation}); and stop only when
$\mathrm{LCB}_r(x)\ge\phist-\varepsilon$.

\begin{proposition}[Replication to certify a gap]\label{prop:rep}
  To decide, with error probability $\le\delta$, whether $\varphi(x)$
  clears a threshold by a true margin $\gamma>0$ (or to order two
  molecules a true gap $\gamma$ apart), it suffices to average
  $r\ge 2\sigma^2\gamma^{-2}\ln(2/\delta)$ independent queries.
\end{proposition}

\begin{proof}
  Then $c_r(\delta)=\sigma\sqrt{2\ln(2/\delta)/r}\le\gamma$, and the
  sub-Gaussian bound places $\varphi(x)$ on the correct side (resp.\ orders
  the means) except on an event of probability $\le\delta$.
\end{proof}

\begin{corollary}[Price of noise]\label{cor:priceofnoise}
  A confidence-$\delta$ robust strategy returns a true
  $\varepsilon$-molecule with probability $\ge1-\mathcal O(\delta)$ using
  \[
    \E[\evals_\varepsilon]\ \lesssim\ \sum_{k=k_0}^{K(\varepsilon)-1}\frac{r_k}{q_k}
    \ \approx\ 2\sigma^2\ln\tfrac2\delta
       \sum_{k=k_0}^{K(\varepsilon)-1}\frac1{\Delta_k^2\,q_k},
    \qquad r_k\approx\frac{2\sigma^2}{\Delta_k^2}\ln\tfrac2\delta,
  \]
  with $\Delta_k$ the fitness gap resolved at level $k$: the noiseless
  cost $\sum_k1/q_k$ (Corollary~\ref{cor:minbudget}) inflated, level by level, by
  $2\sigma^2\Delta_k^{-2}\ln(2/\delta)$.  Convergence weakens from
  almost-sure to high-probability.
\end{corollary}

\begin{remark}[The accuracy $\varepsilon$ enters twice under noise]\label{rem:eps}
  As for every layered bound, $\varepsilon$ is carried by the summation
  range $k_0,\dots,K(\varepsilon)-1$ (Remark~\ref{rem:eps-levels}):
  tightening $\varepsilon$ appends harder near-optimal levels, and
  $N+\sum_k1/q_k$ and $2\sigma^2\ln\tfrac2\delta\sum_k\Delta_k^{-2}/q_k$
  grow without bound as $\varepsilon\to0$ on hard landscapes.  Under noise
  $\varepsilon$ enters a \emph{second} time, through the certification
  margin: stopping on $\mathrm{LCB}_r(x)\ge\phist-\varepsilon$ resolves the
  gap $\Delta_{K-1}=\varphi(x)-(\phist-\varepsilon)$ of the returned
  molecule, which is the lattice spacing for a discrete fitness (so
  $\varepsilon$ acts only through $K(\varepsilon)$, as in our
  \textsc{OneMax} test) but as small as $\varepsilon$ in the worst
  continuous case, where by Proposition~\ref{prop:rep} the top level alone contributes
  $\sigma^2\varepsilon^{-2}\ln(2/\delta)$ to the replication cost.
\end{remark}

Charging every proposal the worst-case $r_k$ is wasteful---most are far
from the threshold and ranked correctly after a couple of queries.  A
\emph{sequential} test (query the incumbent until its bound crosses the
threshold or its budget is spent) concentrates replication on the few
near-threshold decisions, the robust analogue of $M=1$.

\begin{algorithm}[ht]
\caption{Robust EM-SGS (R-EM-SGS), noisy oracle}\label{alg:remsgs}
\begin{algorithmic}[1]
\Require $\varepsilon$; confidence $\delta$; pool $P_0$; depths $\{m_k^\star\}$; explore reps $r_0$; cap $r_{\max}$
\State cache $\cache$ holds $(\text{sum},\text{count})$ per molecule;\ query each of $P_0$ $\,r_0$ times
\Repeat
  \State $b\leftarrow\argmax_{x\in P_t}\widehat\varphi(x)$
  \While{$\widehat\varphi(b)\ge\phist-\varepsilon$ \textbf{and} $\mathrm{LCB}(b)<\phist-\varepsilon$ \textbf{and} $\mathrm{count}(b)<r_{\max}$}
    \State query $b$ once
  \EndWhile
  \If{$\mathrm{LCB}(b)\ge\phist-\varepsilon$}
    \Return $b$ \Comment{confident stop}
  \EndIf
  \State $k\leftarrow\ell(P_t)$;\ $G_t$ from last $m_k^\star$ pools;\ sample $y\sim G_t(\cdot\mid P_t)$;\ query $y$ $\,r_0$ times
  \State $P_{t+1}\leftarrow\topN(P_t\cup\{y\})$ by $\widehat\varphi$;\ refit $G_{t+1}$
\Until{stopped}
\end{algorithmic}
\end{algorithm}

\paragraph{Numerical illustration.}
Adding Gaussian noise $\xi\sim\mathcal N(0,\sigma^2)$ to the same oracle
(as in Section~\ref{sec:num-memory}; $n=30$, $N=10$, $m=1$, target
$\phist-\varepsilon=28$, $z=2$, $r_0=2$, $r_{\max}=60$; $150$ runs) gives
Table~\ref{tab:noise}.  The naive strategy of Section~\ref{sec:ctrl-algo} terminates
quickly but returns molecules of true fitness $27.4,25.3,22.9$---below the
target $28$---as $\sigma=1,2,3$: its success rate collapses
$0.46\!\to\!0.04\!\to\!0$.  Both robust variants are correct; fixed
$r=8$ pays $\sim\!10^3$ evaluations regardless, while the adaptive
\Cref{alg:remsgs} attains the same correctness at $273$--$594$ evaluations by
replicating only to certify a near-target incumbent.  The adaptive cost
grows with $\sigma$ in line with the $\sigma^2$ factor of
Corollary~\ref{cor:priceofnoise} (Figure~\ref{fig:priceofnoise}); at large $\sigma$ a
fixed $(\delta,r_{\max})$ no longer certifies every run, and maintaining
correctness demands more replication---the cost-for-certainty trade-off.

\begin{table}[ht]
\centering
\caption{Noisy \textsc{OneMax}: evaluations, success rate (fraction of runs
  returning a true $\varepsilon$-molecule), and mean true fitness of the
  returned molecule.}
\label{tab:noise}
\renewcommand{\arraystretch}{1.12}
\begin{tabular}{@{}llccc@{}}
\toprule
$\sigma$ & strategy & evals & success & returned $\varphi$ \\
\midrule
$1$ & naive (Section~\ref{sec:ctrl-algo}) & 107  & 0.46 & 27.4 \\
    & robust, fixed $r=8$          & 1037 & 1.00 & 29.0 \\
    & robust, adaptive (Alg.~3)    & 273  & 1.00 & 28.8 \\
\midrule
$2$ & naive                        & 105  & 0.04 & 25.3 \\
    & robust, fixed $r=8$          & 1210 & 1.00 & 29.0 \\
    & robust, adaptive (Alg.~3)    & 371  & 0.96 & 28.6 \\
\midrule
$3$ & naive                        & 102  & 0.00 & 22.9 \\
    & robust, fixed $r=8$          & 1558 & 0.99 & 28.9 \\
    & robust, adaptive (Alg.~3)    & 594  & 0.72 & 28.1 \\
\bottomrule
\end{tabular}
\end{table}

\begin{figure}[ht]
\centering
\begin{tikzpicture}
\begin{axis}[
    width=0.68\textwidth, height=0.4\textwidth,
    xlabel={noise level $\sigma$}, ylabel={evaluations (adaptive)},
    xmin=0.3, xmax=3.2, ymin=200, ymax=650,
    grid=both, major grid style={black!12}, tick align=outside, thick,
    legend pos=north west, legend cell align=left,
]
\addplot[domain=0.4:3.1, samples=60, dashed, gray, thick, forget plot]{242+39*x^2};
\addplot[violet, mark=*, mark size=1.5pt] coordinates {
(0.5,252)(1,273)(1.5,314)(2,371)(2.5,462)(3,594)};
\addlegendentry{adaptive (Alg.~3)}
\addplot[gray,only marks,mark=none] coordinates {(3.05,594)};
\addlegendentry{$\propto\sigma^2$ guide}
\end{axis}
\end{tikzpicture}
\caption{Price of noise: evaluations of the adaptive robust strategy grow
  with $\sigma$ as the $\sigma^2$ replication factor of
  Corollary~\ref{cor:priceofnoise} predicts (dashed guide $242+39\,\sigma^2$).
  Success rates $1.00,1.00,0.99,0.96,0.89,0.72$ for $\sigma=0.5,\dots,3$.}
\label{fig:priceofnoise}
\end{figure}
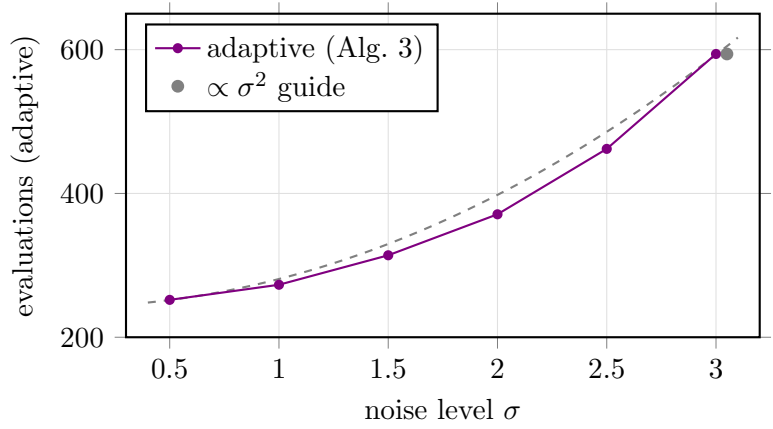

In one line: replace single scoring by \emph{replicated} scoring with
confidence bounds, memoisation by \emph{accumulation}, hard stopping by
$\mathrm{LCB}\ge\phist-\varepsilon$, and allocate repeats \emph{adaptively}
(\Cref{alg:remsgs}); the budget result $M=1$ is unchanged, and the cost gains the
price-of-noise factor of Corollary~\ref{cor:priceofnoise}.

\subsection{Heavy-tailed noise}
\label{sec:ctrl-heavy}

The confidence machinery of Section~\ref{sec:ctrl-noise}---and the price
of noise (Corollary~\ref{cor:priceofnoise})---rests on a \emph{sub-Gaussian}
tail, which lets the sample mean $\widehat\varphi_r$ concentrate
exponentially.  Real scores (docking artefacts, assay spikes) can be
heavier-tailed, and then the sample mean is the wrong tool: a single
outlier shifts it arbitrarily, the Hoeffding radius is mis-calibrated, and
certification either stops falsely or never stops.  Two substitutions
repair this and a third makes it automatic.

\paragraph{Robust location and distribution-free decisions.}
Replace the per-molecule sample mean by a heavy-tail-robust estimator---the
\emph{median-of-means} (split the $r$ queries into
$b=\lceil8\ln(2/\delta)\rceil$ blocks and take the median of the block
means) or a trimmed mean / Catoni M-estimator~\cite{Catoni2012,LugosiMendelson2019};
these regain sub-Gaussian-rate radii assuming only a finite variance.  When
even the variance is infinite, a sign/rank test makes the decisions
\emph{distribution-free}: to certify $\varphi(x)\ge T$ under symmetric
noise, count $c=\#\{\text{queries}>T\}$ and stop once the binomial lower
bound on $p=\Prob(\tilde\varphi(x)>T)$ clears $\tfrac12$; this needs no
moments at all and tolerates an infinite mean.

\begin{proposition}[Replication across noise regimes]\label{prop:heavy}
  To certify a true margin $\gamma>0$ with error $\le\delta$, the number of
  queries $r$ required scales as
  \[
  \begin{array}{ll}
   \text{(a) sub-Gaussian (sample mean):}
     & r\ \gtrsim\ \sigma^2\gamma^{-2}\ln(2/\delta)\quad(\text{Proposition~\ref{prop:rep}});\\[3pt]
   \text{(b) finite variance }\sigma^2\text{ (median-of-means):}
     & r\ \gtrsim\ \sigma^2\gamma^{-2}\ln(2/\delta)\quad(\text{larger constant});\\[3pt]
   \text{(c) finite }(1{+}\alpha)\text{-moment }\nu,\ 0<\alpha<1:
     & r\ \gtrsim\ \bigl(\nu\,\gamma^{-(1+\alpha)}\bigr)^{1/\alpha}\ln(2/\delta);\\[3pt]
   \text{(d) symmetric, sign test (no moments):}
     & r\ \gtrsim\ \dfrac{\ln(2/\delta)}{2\,(p_\gamma-\tfrac12)^2},\quad p_\gamma=\Prob(\xi<\gamma).
  \end{array}
  \]
\end{proposition}

\begin{proof}[Proof sketch]
  (a)~is Proposition~\ref{prop:rep}.  (b)~Median-of-means with
  $b=\lceil8\ln(2/\delta)\rceil$ blocks deviates by at most
  $\sigma\sqrt{32\ln(2/\delta)/r}$ with probability $\ge1-\delta$ under a
  finite variance~\cite{LugosiMendelson2019}; set this $\le\gamma$.
  (c)~A trimmed mean tuned to $\delta$ deviates by
  $\lesssim\nu^{1/(1+\alpha)}(\ln(2/\delta)/r)^{\alpha/(1+\alpha)}$ under a
  finite $(1{+}\alpha)$-moment~\cite{LugosiMendelson2019}; invert.
  (d)~The count $c$ is $\mathrm{Binomial}(r,p_\gamma)$ with
  $p_\gamma=\Prob(\xi>-\gamma)=\Prob(\xi<\gamma)$ by symmetry, and
  Hoeffding's inequality certifies $p_\gamma>\tfrac12$ once
  $r\ge\ln(2/\delta)/\bigl(2(p_\gamma-\tfrac12)^2\bigr)$.
\end{proof}

\begin{remark}[Price of heavy tails, and what to change]\label{rem:heavy}
  The price-of-noise factor $2\sigma^2\Delta_k^{-2}\ln(2/\delta)$ of
  Corollary~\ref{cor:priceofnoise} survives a heavy but finite-variance tail
  with only a worse constant (case (b)); a genuinely infinite variance
  degrades the exponent (case (c)); and the distribution-free route (d)
  trades the fitness gap $\Delta_k$ for the probability gap
  $p_{\Delta_k}-\tfrac12$, costly when the noise density at the threshold
  is low but valid for any tail.  Operationally, three changes suffice:
  estimate a \emph{robust} scale (MAD or interquartile range) in place of
  the sample variance when forming $\widehat\sigma$; \emph{diagnose} the
  tail online---e.g.\ a ratio $\mathrm{SD}/(1.48\,\mathrm{MAD})\gg1$ flags
  a heavy tail---and switch from the cheap mean to a robust estimator or
  the sign test when it is heavy; and, when the noise is \emph{asymmetric}
  (so neither mean nor median equals $\varphi$), define the design objective
  as a fixed quantile of the score distribution, with the levels
  $1,\dots,K(\varepsilon)$ and escape probabilities $q_k$ taken relative to
  that quantile.  These prescriptions are tested numerically below.
\end{remark}

\paragraph{Numerical test.}
We stress the certifier on the noisy \textsc{OneMax} oracle of
Section~\ref{sec:num-memory}, replacing the Gaussian noise by a Student-$t$
variate of $\nu$ degrees of freedom, each scaled to a common interquartile
range (that of $\mathcal N(0,2)$) so that the bulk is fixed and only the
\emph{tail} varies (Figure~\ref{fig:ht-density}); $\nu$ ranges from $\infty$
(Gaussian) through $\nu=2$ (variance just infinite) to $\nu=1$ (Cauchy,
infinite mean).  Holding the controller of Section~\ref{sec:ctrl-algo} fixed
($\varepsilon=2$, $\delta=0.1$, $r_0=2$; $150$ runs), we change only the
location estimate and the certification rule: the sub-Gaussian sample mean,
a median-of-means estimate with a robust (MAD) scale radius, and the
distribution-free sign test of Proposition~\ref{prop:heavy}(d).

\begin{figure}[ht]
\centering
\begin{tikzpicture}
\begin{axis}[
    width=0.66\textwidth, height=0.4\textwidth,
    xlabel={noise $\xi$}, ylabel={density (log scale)}, ymode=log,
    xmin=-7, xmax=7, ymin=3e-4, ymax=0.32,
    grid=both, major grid style={black!12}, minor grid style={black!6},
    tick align=outside, thick, legend pos=north east, legend cell align=left,
]
\addplot[black, thick] coordinates {
(-7.0,0.0004)(-6.6,0.0009)(-6.2,0.0016)(-5.8,0.0030)(-5.4,0.0052)(-5.0,0.0088)(-4.6,0.0142)(-4.2,0.0220)(-3.8,0.0328)(-3.4,0.0470)(-3.0,0.0648)(-2.6,0.0857)(-2.2,0.1089)(-1.8,0.1330)(-1.4,0.1561)(-1.0,0.1760)(-0.6,0.1907)(-0.2,0.1985)(0.0,0.1995)(0.2,0.1985)(0.6,0.1907)(1.0,0.1760)(1.4,0.1561)(1.8,0.1330)(2.2,0.1089)(2.6,0.0857)(3.0,0.0648)(3.4,0.0470)(3.8,0.0328)(4.2,0.0220)(4.6,0.0142)(5.0,0.0088)(5.4,0.0052)(5.8,0.0030)(6.2,0.0016)(6.6,0.0009)(7.0,0.0004)};
\addlegendentry{$\nu=\infty$ (Gaussian)}
\addplot[teal, thick] coordinates {
(-7.0,0.0053)(-6.6,0.0065)(-6.2,0.0080)(-5.8,0.0098)(-5.4,0.0122)(-5.0,0.0154)(-4.6,0.0195)(-4.2,0.0249)(-3.8,0.0321)(-3.4,0.0416)(-3.0,0.0540)(-2.6,0.0701)(-2.2,0.0904)(-1.8,0.1148)(-1.4,0.1423)(-1.0,0.1700)(-0.6,0.1932)(-0.2,0.2066)(0.0,0.2084)(0.2,0.2066)(0.6,0.1932)(1.0,0.1700)(1.4,0.1423)(1.8,0.1148)(2.2,0.0904)(2.6,0.0701)(3.0,0.0540)(3.4,0.0416)(3.8,0.0321)(4.2,0.0249)(4.6,0.0195)(5.0,0.0154)(5.4,0.0122)(5.8,0.0098)(6.2,0.0080)(6.6,0.0065)(7.0,0.0053)};
\addlegendentry{$\nu=3$}
\addplot[orange, thick] coordinates {
(-7.0,0.0077)(-6.6,0.0088)(-6.2,0.0102)(-5.8,0.0118)(-5.4,0.0139)(-5.0,0.0165)(-4.6,0.0197)(-4.2,0.0239)(-3.8,0.0293)(-3.4,0.0364)(-3.0,0.0459)(-2.6,0.0586)(-2.2,0.0757)(-1.8,0.0986)(-1.4,0.1278)(-1.0,0.1621)(-0.6,0.1956)(-0.2,0.2174)(0.0,0.2204)(0.2,0.2174)(0.6,0.1956)(1.0,0.1621)(1.4,0.1278)(1.8,0.0986)(2.2,0.0757)(2.6,0.0586)(3.0,0.0459)(3.4,0.0364)(3.8,0.0293)(4.2,0.0239)(4.6,0.0197)(5.0,0.0165)(5.4,0.0139)(5.8,0.0118)(6.2,0.0102)(6.6,0.0088)(7.0,0.0077)};
\addlegendentry{$\nu=1.5$}
\addplot[red, thick] coordinates {
(-7.0,0.0084)(-6.6,0.0095)(-6.2,0.0107)(-5.8,0.0121)(-5.4,0.0139)(-5.0,0.0160)(-4.6,0.0187)(-4.2,0.0221)(-3.8,0.0264)(-3.4,0.0321)(-3.0,0.0397)(-2.6,0.0500)(-2.2,0.0645)(-1.8,0.0849)(-1.4,0.1136)(-1.0,0.1523)(-0.6,0.1970)(-0.2,0.2309)(0.0,0.2360)(0.2,0.2309)(0.6,0.1970)(1.0,0.1523)(1.4,0.1136)(1.8,0.0849)(2.2,0.0645)(2.6,0.0500)(3.0,0.0397)(3.4,0.0321)(3.8,0.0264)(4.2,0.0221)(4.6,0.0187)(5.0,0.0160)(5.4,0.0139)(5.8,0.0121)(6.2,0.0107)(6.6,0.0095)(7.0,0.0084)};
\addlegendentry{$\nu=1$ (Cauchy)}
\end{axis}
\end{tikzpicture}
\caption{Student-$t$ noise densities used in the test, scaled to a common
  interquartile range (that of $\mathcal N(0,2)$).  On the logarithmic axis
  the Gaussian tail plunges quadratically while the heavier-tailed densities
  decay only polynomially---at $\xi=\pm7$ the Cauchy density is about
  $20\times$ the Gaussian.  Only the tail (and the slightly taller
  leptokurtic peak) differs.}
\label{fig:ht-density}
\end{figure}
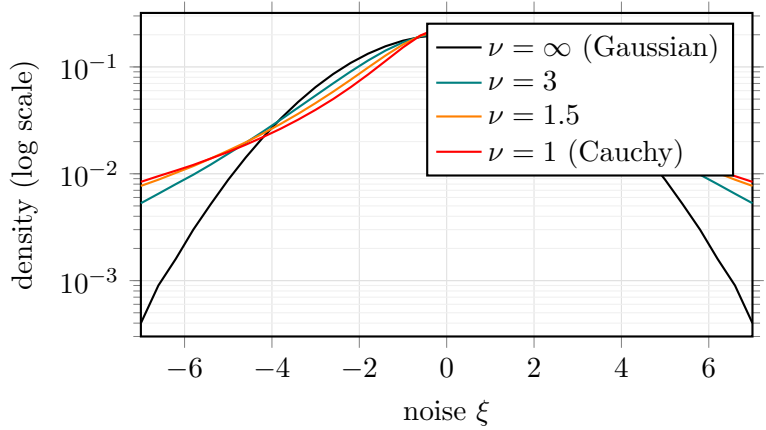

Table~\ref{tab:ht} reports the outcome.  The mean-based certifier's success
rate---the fraction of runs that truly return an $\varepsilon$-molecule---
falls from $0.95$ under Gaussian noise to $0$ at Cauchy: an outlier inflates
a molecule's mean and the variance radius is mis-calibrated, so it stops on
noise rather than fitness.  A robust \emph{point} estimate does not rescue
it---median-of-means with a robust-scale radius \emph{under-covers},
ignoring the very outliers that drive the mean's uncertainty, and is no
better.  Only the sign test, certifying the median crossing through a
binomial bound, stays correct across every tail (success $\ge0.98$ including
Cauchy), realising Proposition~\ref{prop:heavy}(d).  Its price is exactly the
replication that proposition predicts: the queries grow from $594$
(Gaussian) to $3368$ (Cauchy) as the probability gap $p_\Delta-\tfrac12$ at
the threshold shrinks (Figure~\ref{fig:ht-succ})---the heavy-tailed face of
the price of noise.

\begin{table}[ht]
\centering
\caption{Evaluations and success (in parentheses) under Student-$t$ noise
  on \textsc{OneMax} ($\varepsilon{=}2$, $150$ runs).  Mean and
  median-of-means certification collapse as the tail fattens; the
  distribution-free sign test stays correct, at a query cost that grows with
  the tail.}
\label{tab:ht}
\renewcommand{\arraystretch}{1.12}
\begin{tabular}{@{}lccc@{}}
\toprule
tail $\nu$ & sample mean & median-of-means & sign test \\
\midrule
$\infty$ (Gaussian) & 465\,(0.95) & 412\,(0.95) & 594\,(\textbf{0.99}) \\
$4$                 & 486\,(0.63) & 390\,(0.56) & 698\,(\textbf{0.99}) \\
$3$                 & 463\,(0.43) & 336\,(0.30) & 727\,(\textbf{0.99}) \\
$2$ (var.\ $\infty$)& 313\,(0.19) & 209\,(0.11) & 921\,(\textbf{0.99}) \\
$1.5$               & 382\,(0.05) & 140\,(0.03) & 1552\,(\textbf{0.99}) \\
$1$ (Cauchy)        & 160\,(0.00) & 63\,(0.00)  & 3368\,(\textbf{0.98}) \\
\bottomrule
\end{tabular}
\end{table}

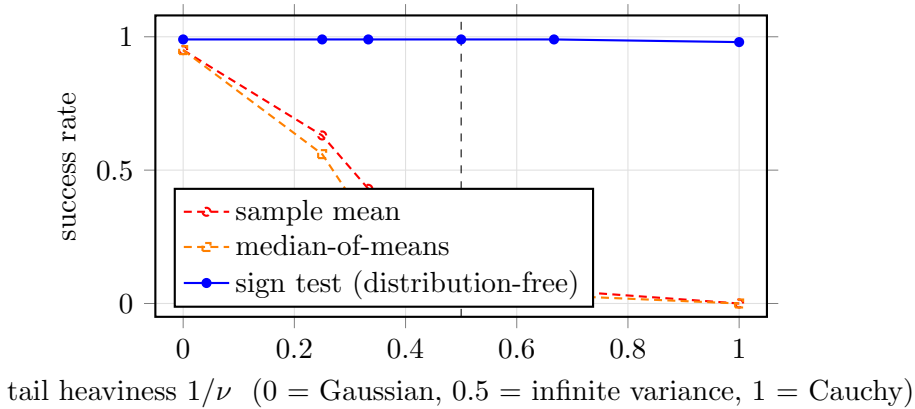
\begin{figure}[ht]
\centering
\begin{tikzpicture}
\begin{axis}[
    width=0.66\textwidth, height=0.38\textwidth,
    xlabel={tail heaviness $1/\nu$\ \ (0 = Gaussian, $0.5$ = infinite variance, $1$ = Cauchy)},
    ylabel={success rate}, xmin=-0.05, xmax=1.05, ymin=-0.05, ymax=1.08,
    grid=both, major grid style={black!12}, tick align=outside, thick,
    legend pos=south west, legend cell align=left,
]
\addplot[red, mark=o, mark size=1.5pt, densely dashed] coordinates {
(0,0.95)(0.25,0.63)(0.333,0.43)(0.5,0.19)(0.667,0.05)(1,0.00)};
\addlegendentry{sample mean}
\addplot[orange, mark=square, mark size=1.5pt, densely dashed] coordinates {
(0,0.95)(0.25,0.56)(0.333,0.30)(0.5,0.11)(0.667,0.03)(1,0.00)};
\addlegendentry{median-of-means}
\addplot[blue, mark=*, mark size=1.5pt] coordinates {
(0,0.99)(0.25,0.99)(0.333,0.99)(0.5,0.99)(0.667,0.99)(1,0.98)};
\addlegendentry{sign test (distribution-free)}
\addplot[black, dashed, thin, forget plot] coordinates {(0.5,-0.05)(0.5,1.08)};
\end{axis}
\end{tikzpicture}
\caption{Success rate versus tail heaviness on noisy \textsc{OneMax}.  Mean
  and median-of-means certification collapse beyond a mild tail; the
  distribution-free sign test holds $\ge0.98$ throughout, including the
  infinite-mean Cauchy case.  The vertical line marks $\nu=2$ (infinite
  variance).}
\label{fig:ht-succ}
\end{figure}

\section{Concluding Remarks}
\label{sec:concl}

We have given a self-contained convergence, hitting-time, and
evaluation-cost theory for closed-loop generative selection, organised
around a single dichotomy: results that hold for any oracle, and results
that quantify what noise costs.  In the first group the layered
hitting-time bound, its memory-dependent refinements and the
non-monotone exit-time theory, and the reduction of the proposal budget to
its evaluation-optimal corner $M=1$ form a coherent picture in which the
landscape difficulty $\sum_k 1/q_k$ is the fundamental quantity.  In the
second group, sub-Gaussian noise inflates this difficulty by a
price-of-noise factor $\sigma^2/\Delta^2$, and heavy tails replace the
sample mean by robust or distribution-free certifiers whose cost degrades
gracefully with the tail.

The theory developed here is deliberately abstract: it does not depend on
the specific form of the generative model, the fitness landscape, or the
noise distribution beyond broad qualitative assumptions (uniform
positivity, sub-Gaussian or finite-variance tails).  This abstraction is a
strength---the bounds apply to any instantiation of generative
selection---but also a limitation, because the tightness of every bound
hinges on the escape probabilities $q_k$ and their memory-dependent
refinements $q_k^m$, which are algorithm-specific quantities.  Three
problems stand out as particularly pressing.

\paragraph{1. Quantifying and maximising the escape probabilities.}
Every bound in this paper is only as good as the lower bounds on the escape
probabilities $q_k$ and their memory-dependent versions $q_k^m$.  Deriving
$q_k^m$ from the generative model---as a function of the pool size, the
number of fine-tuning steps, the molecular representation, and the geometry
of the chemical-similarity graph---and proving matching lower bounds that
capture the benefit of memory, in the spirit of Sudholt~\cite{Sudholt2013},
is the central quantitative challenge.  A sharp answer would turn the
qualitative guarantees here into predictive, algorithm-specific running
times.  We suspect that for neural generative models (e.g.\ variational
autoencoders, normalising flows, or diffusion models trained on molecular
graphs), the escape probabilities are governed by the overlap between the
model's latent representation and the fitness-relevant degrees of freedom,
but a rigorous characterisation remains open.

\paragraph{2. Memory, monotonicity, and composition-homogeneity.}
The memory hierarchy holds under the monotone-learning condition, and the
non-monotone theory locates an optimal depth at the first peak of the
learning profile; but our numerical study shows that a naive windowed
learner is only approximately composition-homogeneous, so the sharp
per-level results and the promised gain of level-adaptive memory do not
fully materialise.  Designing a learner that is at once strongly
non-monotone and composition-homogeneous---so that adaptive depth pays
off---and characterising the optimal depth and the three-way trade-off
among pool size, proposal budget, and memory remain open.  One promising
direction is to replace the sliding-window average with a
forgetting-factor or kernel-weighted estimate that downweights older pools
smoothly, which may restore composition-homogeneity while preserving the
bias--variance trade-off that makes finite memory beneficial.

\paragraph{3. Robust certification under realistic noise and in continuous
spaces.}
Our heavy-tailed treatment assumes symmetric noise and a finite chemical
space.  Several extensions are needed for practical drug design.
Anytime-valid sequential tests would tighten the query cost and remove the
manual tuning of confidence parameters.  Asymmetric heavy tails call for a
quantile objective, against which the levels and escape probabilities must
be redefined.  Continuous latent representations $\X\subseteq\R^p$ turn the
memory-dependent kernel into a history-dependent Markov kernel on an
uncountable state space, to which the stability theory of Meyn and
Tweedie~\cite{MeynTweedie1993} applies only after substantial adaptation to
learned kernels.  Together these would extend the present theory to the
oracles and spaces of practical drug design, where the search space is
continuous (e.g.\ latent codes of a molecular autoencoder) and the oracle
is a noisy multi-objective assay with unknown tail behaviour.

Beyond these specific open problems, we believe the broader contribution of
this paper is a conceptual framework for analysing adaptive generative
loops that separates the learning dynamics (memory, escape probabilities)
from the search dynamics (elitism, hitting times) and identifies the
landscape difficulty $\sum_k 1/q_k$ as the fundamental quantity that
governs performance.  We hope this framework will serve as a foundation for
tighter, algorithm-specific analyses and for the principled design of
generative-selection algorithms in drug discovery and beyond.

\appendix

\section{Table of Notation}
\label{app:notation}

Table~\ref{tab:notation} collects the symbols used throughout the paper,
grouped by role.  Page-level cross-references point to the defining
equation, definition, or assumption.

\renewcommand{\arraystretch}{1.25}
\begin{small}
\begin{longtable}{@{}p{0.20\textwidth} p{0.62\textwidth} p{0.12\textwidth}@{}}
\caption{Summary of notation.}\label{tab:notation}\\
\toprule
\textbf{Symbol} & \textbf{Meaning} & \textbf{Ref.}\\
\midrule
\endfirsthead
\multicolumn{3}{@{}l}{\small\itshape Table~\ref{tab:notation} (continued)}\\
\toprule
\textbf{Symbol} & \textbf{Meaning} & \textbf{Ref.}\\
\midrule
\endhead
\midrule
\multicolumn{3}{r@{}}{\small\itshape continued on next page}\\
\endfoot
\bottomrule
\endlastfoot

\multicolumn{3}{@{}l}{\textit{Spaces and fitness}}\\
$\X$ & Finite chemical (search) space; $|\X|$ may exceed $10^{12}$ & \S\ref{sec:framework}\\
$\varphi$ & Scalar fitness functional $\varphi\colon\X\to\R$ (higher is better) & \S\ref{sec:framework}\\
$\phist$ & Global optimum $\phist=\max_{x\in\X}\varphi(x)$ & \S\ref{sec:framework}\\
$\varepsilon$ & Optimality tolerance & \S\ref{sec:framework}\\
$\Neps$ & Target set $\{x:\varphi(x)\ge\phist-\varepsilon\}$ & \S\ref{sec:framework}\\
$N$ & Elite pool size (number of retained molecules) & \S\ref{sec:framework}\\
$\Omn$ & Population space, the $N$-element subsets of $\X$ & \S\ref{sec:framework}\\
$P,P_t$ & A population (elite pool); the pool at round $t$ & Alg.~1\\
$f(P)$ & Elite fitness $f(P)=\max_{x\in P}\varphi(x)$ & \S\ref{sec:framework}\\
$M$ & Proposal budget (candidates generated per round) & Alg.~1\\
$C_t$ & Set of $M$ candidates proposed at round $t$ & Alg.~1\\

\midrule
\multicolumn{3}{@{}l}{\textit{Fitness levels and absorbing structure}}\\
$v_1<\dots<v_L$ & Distinct fitness values on $\X$ & \Cref{def:levels}\\
$L_j$ & Level set $\{x:\varphi(x)=v_j\}$ & \Cref{def:levels}\\
$K$ & Index of the first level meeting the $\varepsilon$-target & \Cref{def:levels}\\
$f_k$ & Fitness value at level $k$, $f_k=v_k$ ($k<K$) & \Cref{def:levels}\\
$\Akp$ & Escape set $\{y:\varphi(y)>f_k\}=\bigcup_{j>k}L_j$ & \Cref{def:levels}\\
$\ell(P)$ & Level of population $P$ (rank of best molecule, capped at $K$) & \Cref{def:levels}\\
$\Oms$ & Absorbing set $\{P:\ell(P)=K\}$ & \Cref{def:absorbing}\\
$\Omt$ & Transient set $\Omn\setminus\Oms$ & \Cref{def:absorbing}\\
$\tau_\varepsilon$ & Hitting time $\inf\{t:P_t\in\Oms\}$ & \Cref{def:absorbing}\\
$\tau_k$ & First time the level reaches $k$, $\inf\{t:\ell(P_t)\ge k\}$ (so $\tau_\varepsilon=\tau_{K(\varepsilon)}$) & \Cref{lem:sojourn}\\

\midrule
\multicolumn{3}{@{}l}{\textit{Generative kernels and Markov structure}}\\
$G_t$ & Generative model / transition kernel at round $t$ & Alg.~1\\
$\G$ & Space of probability kernels on $\X$ & \Cref{def:augmented}\\
$U$ & Uniform distribution on $\X$ (non-adaptive baseline) & \Cref{ass:positivity}\\
$\alpha$ & Uniform-positivity floor, $G_t(x\mid\cdot)\ge\alpha$ & \Cref{ass:positivity}\\
$\beta$ & Per-round absorption probability $1-(1-\alpha)^M$ & \Cref{thm:base-conv}\\
$Z_t$ & Augmented state $(P_t,G_t)$ & \Cref{def:augmented}\\
$\mathbf{H}_t$ & History process $(P_0,\dots,P_t)$ & \Cref{def:full}\\
$S_\infty$ & History state space $\bigcup_{n\ge1}\Omn^n$ & \Cref{def:full}\\
$F$ & Full-memory map $\bigcup_t\Omn^{t+1}\to\G$ & \Cref{def:full}\\
$G(\cdot)$ & $m$-step memory map $\Omn^m\to\G$ & \Cref{def:m-kernel}\\
$S_t^m$ & $m$-step window $(P_{t-m+1},\dots,P_t)$ & \Cref{def:m-kernel}\\
$m$ & Memory depth (number of past pools used by $G_t$) & \Cref{def:m-kernel}\\

\midrule
\multicolumn{3}{@{}l}{\textit{Escape probabilities and hitting-time bounds}}\\
$q_k$ & Single-proposal escape probability at level $k$ (static) & \Cref{def:qk}\\
$p_k$ & Population escape probability $1-(1-q_k)^M$ & \Cref{def:qk}\\
$q_k^0,p_k^0$ & Baseline (uniform $G_t\equiv U$): $q_k^0=\alpha|\Akp|$ & \Cref{cor:chain}\\
$q_k^\infty,p_k^\infty$ & Full-memory escape probabilities & \Cref{def:full-qk}\\
$q_k^m,p_k^m$ & $m$-step escape probabilities & \Cref{def:m-qk}\\
$D,\,D^m$ & Landscape difficulty $\sum_k 1/q_k$; $m$-step version $\sum_k 1/q_k^m$ & \Cref{def:difficulty}\\

\midrule
\multicolumn{3}{@{}l}{\textit{Learning, regret, and the saturation model}}\\
$q_k^{(s)}$ & Escape probability at the $s$-th round spent at level $k$ & \S\ref{sec:full}\\
$r_s$ & Per-round relative regret, $q_k^{(s)}\ge q_k^\infty(1-r_s)$ & \Cref{thm:online}\\
$R_\infty$ & Cumulative regret $\sum_{s\ge0}r_s$ & \Cref{thm:online}\\
$c_k$ & Learning rate in the saturation model & \Cref{def:sat}\\
$A_k(m)$ & Transient sojourn term $\sum_{s<m}\prod_{j<s}(1-p_k^j)$ & \Cref{prop:sojourn-sat}\\
$B_k(m)$ & Tail weight $\prod_{j<m}(1-p_k^j)$ & \Cref{prop:sojourn-sat}\\
$\Delta_k(m)$ & Sojourn saving from $m$ steps of memory & \Cref{prop:sojourn-sat}\\

\midrule
\multicolumn{3}{@{}l}{\textit{Multi-objective extension}}\\
$\Phi$ & Vector objective $\Phi=(\Phi_1,\dots,\Phi_d)\colon\X\to\R^d$ & \S\ref{sec:multiobjective}\\
$\succeq$ & Componentwise (Pareto) dominance on $\R^d$ & \S\ref{sec:multiobjective}\\
$r$ & Hypervolume reference point, $r_i<\Phi_i(x)\ \forall x,i$ & \Cref{def:hv}\\
$\lambda_d$ & $d$-dimensional Lebesgue measure & \Cref{def:hv}\\
$\mathcal{H}(P)$ & Hypervolume indicator of population $P$ & \Cref{def:hv}\\
$q_k^{\mathrm{HV}}$ & Hypervolume escape probability at level $k$ & \Cref{def:hv-qk}\\
$\mathcal{N}_\varepsilon^{\mathrm{HV}}$ & Hypervolume target set & \S\ref{sec:multiobjective}\\

\midrule
\multicolumn{3}{@{}l}{\textit{Numerical illustration (\Cref{sec:num-setup})}}\\
$n$ & Bit-string length; test space $\X=\{0,1\}^n$ & \Cref{sec:num-setup}\\
$\widehat\theta_{t,i}$ & Empirical frequency of bit $i$ over pools in memory & \Cref{sec:num-setup}\\
$\theta_{t,i}$ & Floored marginal $\alpha'+(1-2\alpha')\widehat\theta_{t,i}$ & \Cref{sec:num-setup}\\
$\alpha'$ & Per-coordinate exploration floor & \Cref{sec:num-setup}\\
$m^\star$ & Empirically optimal memory depth & \Cref{sec:num-memory}\\
$\sigma$ & Standard deviation of evaluation noise & \Cref{sec:num-memory}\\

\end{longtable}
\end{small}

\section*{Acknowlegdement}
The authors acknowledge the use of AI-based tools for language refinement of the manuscript and for assistance during software development. All AI-generated content was reviewed and verified by the authors, who take full responsibility for the final work.

\section*{Data and code availability}
\addcontentsline{toc}{section}{Data and code availability}
The numerical illustrations are fully reproducible.  The OneMax simulation,
the scripts that regenerate each figure's data, and the verification scripts
for the quantitative claims accompany the source, together with the
externalised figure data and the rendered figures.  No external data sets were
used.

\section*{Funding}
\addcontentsline{toc}{section}{Funding}
This work was supported by the Cluster of Excellence MATH+.

\section*{Competing interests}
\addcontentsline{toc}{section}{Competing interests}
The authors declare no competing interests.


\end{document}